\documentclass[reqno,11pt]{amsart}
\usepackage{amsmath,amsthm,amssymb,amsfonts,color,pifont}
\usepackage{framed}
\usepackage[colorlinks=true,urlcolor=blue,citecolor=red,linkcolor=blue,linktocpage,pdfpagelabels,bookmarksnumbered,bookmarksopen]{hyperref}
\usepackage{pgf,tikz}
\usepackage{mathrsfs}
\usetikzlibrary{arrows}

\newtheorem{theorem}{Theorem}[section]
\newtheorem{corollary}[theorem]{Corollary}
\newtheorem{lemma}[theorem]{Lemma}

\newtheorem{proposition}[theorem]{Proposition}
\theoremstyle{definition}
\newtheorem{remark}[theorem]{Remark}
\newtheorem{definition}[theorem]{Definition}

\def\A{\mathcal{A}}
\def\Cap{{\rm Cap}}
\def\wFaQn{\widehat{\mathcal F}_{n,\alpha}}
\def\wFaQ{\widehat{\mathcal F}_{n,\alpha}}
\def\wFQ{\widehat{\mathcal F}_{n,1}}

\def\R{\mathbb R}
\def\FaQ{\mathcal F_{Q,\alpha}}
\def\FzQ{\mathcal{F}_{Q,0}}

\def\G{\mathcal G_\Lambda}

\def\H{\mathcal{H}}
\def\Hu{\H^1}

\def\LM#1{\hbox{\vrule width.2pt \vbox to#1pt{\vfill \hrule width#1pt
height.2pt}}}
\def\LL{{\mathchoice {\>\LM7\>}{\>\LM7\>}{\,\LM5\,}{\,\LM{3.35}\,}}}
\def\restr{{\LL}}

\def\I{\mathcal{I}}
\def\Iz{\I_0}
\def\Ia{\I_\alpha}
 
\def\les{\lesssim}
\def\ges{\gtrsim}
\def\eps{\varepsilon}
\def\xu{x_1^\eps}
\def\xd{x_2^\eps}

\def\lt{\left}
\def\rt{\right}
\def\vphi{\varphi}
\def\diam{\textup{diam}}
\def\N{\mathbb{N}}
\author{Michael Goldman}
\address{LJLL, Universit\'e Paris Diderot, CNRS, UMR 7598, Paris, France}
\email{goldman@math.univ-paris-diderot.fr}

\author{Matteo Novaga}
\address{Department of Mathematics, University of Pisa, 56100 Pisa, Italy} 
\email{matteo.novaga@unipi.it} 

\author{Berardo Ruffini}
\address{Institut Montpelli\'erain Alexander Grothendieck, University of Montpellier, 34095 Montpellier Cedex 5, France}
\email{berardo.ruffini@umontpellier.fr}

\numberwithin{equation}{section}

\title[A long-range-isoperimetric problem
  under  under a convexity constraint]{On minimizers of an 
	isoperimetric problem
	with long-range interactions under a convexity constraint}

\begin{document}

\begin{abstract}
We study a variational problem 
modeling the behavior at equilibrium of charged liquid drops under convexity constraint. After proving well-posedness of the model, 
we show $C^{1,1}$-regularity of minimizers for the Coulombic interaction in dimension two. As a by-product we obtain that balls are the unique minimizers for small charge. 
Eventually, we study the asymptotic behavior of minimizers, as the charge goes to infinity. 
\end{abstract}

\maketitle


\section{Introduction}\label{section1}

In this paper we are interested in the existence and regularity of minimizers 
of the following problem:
\begin{equation}\label{problema}
\min\left\{\mathcal F_{Q,\alpha}(E)
:\ E\subset\R^N\,\text{convex body},\ |E|=V\right\}
\end{equation}
where, for $E\subset\R^N$, $V, Q>0$ and $\alpha\in [0,N)$, we have set
\begin{equation}\label{Falpha}
\mathcal F_{Q,\alpha}(E):= P(E)+Q^2\mathcal I_\alpha (E).
\end{equation}
Here $P(E):=\mathcal{H}^{N-1}(\partial E)$ stands for the perimeter of $E$ and, letting $\mathcal{P}(E)$ be the set of probability measures supported on the closure of $E$, we set for $\alpha\in (0,N)$,
\begin{equation}\label{Ialpha}
\I_\alpha(E):=\inf_{\mu\in \mathcal{P}(E)}
\int_{E\times E}\frac{d\mu(x)\,d\mu(y)}{|x-y|^\alpha},
\end{equation}
and for $\alpha=0$,
\begin{equation}\label{Izero}
\I_0(E):=\inf_{\mu\in \mathcal{P}(E)}
\int_{E\times E}\log\left(\frac{1}{|x-y|}\right)d\mu(x)\,d\mu(y). 
\end{equation}
Notice that, up to rescaling, we can assume, as we shall do for the rest of the paper, that $V=1$.

Starting from the seminal work of Lord Rayleigh \cite{Ray} (in the Coulombic case $N=3$, $\alpha=1$), the functional \eqref{Falpha} has been extensively studied in the physical literature to model the shape
of charged liquid drops (see \cite{gnrI} and the references therein). In particular, it is known that the ball is a {\em linearly} stable critical point 
for \eqref{problema} if the charge $Q$ is not too large (see for instance \cite{FonFri}).
However, quite surprisingly, the authors showed in \cite{gnrI} that, without the convexity constraint, \eqref{Falpha} never admits minimizers 
under volume constraint for any $Q>0$ and $\alpha<N-1$. In particular, this implies that in this model 
a charged drop is always {\em nonlinearly} unstable. This result is 
in sharp contrast with experiments (see for instance 
\cite{Zel,Tay}), where there is evidence of stability of the ball for small charges. This suggests
that the energy $\mathcal F_{Q,\alpha}(E)$ does not include
all the physically relevant contributions. 

As shown in \cite{gnrI}, a possible way to gain well-posedness of the problem is requiring some extra regularity of the admissible sets. In this paper, we consider an alternative type of constraint, 
namely the convexity of admissible sets. This assumption seems reasonable as long as the minimizers remain strictly convex, that is for small enough charges.
Let us point out that in \cite{MurNov}, still another regularizing  mechanism is proposed. There, well-posedness is obtained by adding an entropic 
term which prevents charges to concentrate too much on the boundary of $E$. 
We point out that it has been recently shown in  \cite{murnovruf} that in the borderline case $\alpha=1$, $N=2$ such a regularization is not needed for the model to be well-posed.   
For a more exhaustive discussion about the physical motivations and the literature on related problems we refer to the papers \cite{MurNov,gnrI}. 

Using the compactness properties of convex sets, our first result is the existence of minimizers for every charge $Q>0$.
\begin{theorem}\label{teoesist}
 For every $\alpha \in [0,N)$ and every $Q$, \eqref{problema} admits a minimizer.
\end{theorem}
We then  study the regularity of minimizers. As often in variational problems with convexity constraints, regularity (or singularity) of minimizers is hard to deal with in dimension larger than two (see \cite{lamboley,lamboley2}). We thus restrict ourselves to $N=2$. Since our analysis strongly
uses the regularity of {\em equilibrium measures} (i.e. the minimizer of \eqref{Ialpha}), we are further reduced to study the case $\alpha=N-2$ (that is $\alpha=0$ in this case). The second main result of the paper is then
\begin{theorem}\label{teoreg}
 Let $N=2$ and $\alpha=0$, then for every $Q>0$, the minimizers of \eqref{problema} are of class $C^{1,1}$.
\end{theorem}
Since we are able to prove uniform $C^{1,1}$ estimates as $Q$ goes to zero, building upon our previous stability results established in \cite{gnrI}, we get
\begin{corollary}\label{corballs}
 If $N=2$ and $\alpha=0$, for $Q$ small enough, the only minimizers of \eqref{problema} are balls.
\end{corollary}

The proof of Theorem \ref{teoreg} is based on the natural idea of  comparing the minimizers with a competitor made by ``cutting out the angles''. However, the non-local nature of the problem makes the estimates non-trivial. As already mentioned, a crucial point is an estimate on the integrability of the 
equilibrium measures. This is obtained by drawing a connection with harmonic measures (see Section \ref{secdens}). Let us point out\footnote{this was suggested to us by J. Lamboley} that, up to proving the regularity  of the shape functional $\Iz$ and computing its shape derivative, 
one could have obtained a proof of Theorem \ref{teoreg} by applying the abstract regularity result of \cite{lamboley}. Nevertheless, since our proof has a nice geometrical 
flavor and since regularity of $\Iz$ is not known in dimension two (see for instance \cite{jerisonacta,crasta,novruf} for the proof in higher dimension), we decided to keep it.
\\
We remark that, differently from the two-dimensional case, 
when $N=3$ we expect the onset of singularities 
at a critical value $Q_c>0$, with the shape of a spherical cone with a prescribed angle. Such singularities are also observed in experiments and are usually
called {\it Taylor cones} (see \cite{Tay,Zel}). At the moment
we are not able to show the presence of such singularities in our model,
and this will be the subject of future research.\\

Eventually, in Section \ref{seclim}, we study the behavior of the optimal sets when the charge goes to infinity. Even though this regime is less significant from the point of view of 
the applications, we believe that it is still mathematically interesting. Building on $\Gamma-$convergence results, we prove
\begin{theorem}\label{teoconv}
 Let $\alpha\in [0,1)$ and $N\ge 2$. Then, every minimizer $E_Q$ of \eqref{problema} satisfies (up to a rigid motion)
 \[Q^{-\frac{2N(N-1)}{1+(N-1)\alpha}} E_Q\to [0,L_{N,\alpha}]\times\{0\}^{N-1},\]
 where the convergence is in the Hausdorff topology and where 
 \[L_{N,\alpha}:=\lt(\frac{\alpha (N-1) \Ia([0,1])}{N^{(N-2)/(N-1)} \omega_{N-1}^{1/(N-1)}}\rt)^{\frac{(N-1)}{1+\alpha(N-1)}} \textrm{ for } \alpha \in(0,1) \qquad \textrm{and} \qquad L_{N,0}:=\frac{(N-1)^{N-1}}{\omega_{N-1} N^{N-2}},\]
  $\omega_{N}$ being the volume of  the unit ball in $\R^N$. 
 For $\alpha=1$ and $N=2,3$, we have
 \[Q^{-\frac{2(N-1)}{N}}(\log Q)^{-1+1/N}E_Q\to [0,L_{N,1}]\times \{0\}^{N-1},\]
 where
 \[L_{N,1}:=\lt(\frac{4(N-1)}{N^{(N-2)/(N-1)} \omega_{N-1}^{1/(N-1)}}\rt)^{(N-1)/N}.\]
\end{theorem}

\noindent An obvious consequence of this result is that the ball cannot be a minimizer for $Q$ large enough. For a careful analysis of the loss of linear stability of the ball  we refer to \cite{FonFri}.



\medskip

\noindent{\bf Acknowledgements.} The authors wish 
to thank Guido De Philippis, Jimmy Lamboley, Antoine Lemenant  
and Cyrill Muratov for useful discussions 
on the subject of this paper. 
M. Novaga and B. Ruffini were partially  supported  by  the  Italian 
CNR-GNAMPA  and  by  the  University  of  Pisa  via grant PRA-2015-0017.

%
%

\section{Existence of minimizers}\label{secex}

In this section we show that the minimum in \eqref{problema} is achieved. We begin with a simple lemma linking estimates on the energy with estimates on the size of the convex body.

\begin{lemma}\label{phi}
 Let $N\ge 2$, and $\lambda_1,.., \lambda_N>0$.  Letting $E:=\prod_{i=1}^N [0,\lambda_i]$, $V:=|E|$ and $\Phi:=V^{-\frac{N-2}{N-1}}P(E)$, it holds that\footnote{here and in the rest of the paper, we write $f\les g$ if there exists $C>0$ such that $f\le  C g$. If $f\les g$ and $g\les f$, we will simply write $f\sim g$}  
 \begin{equation}\label{estimlambdageom}
  \max_{i} \lambda_i\les \Phi^{N-1} \qquad \textrm{and} \qquad \min_{i} \lambda_i\sim V^{\frac{1}{N-1}}\Phi^{-1},
 \end{equation}
where the involved constants depend only on the dimension. Moreover, letting $i_{\max}$ be such that $\lambda_{i_{\max}}=\max_i \lambda_i$, it holds for $\alpha>0$,
\begin{equation}\label{estimlambdaalpha}
 \lambda_{i_{\max}}\ges \Ia(E)^{-1/\alpha} \qquad \textrm{and} \qquad  \lambda_i \les \Ia(E)^{1/\alpha}\Phi^{N-2} V^{\frac{1}{N-1}} \quad \text{ for } i\neq i_{\max},
\end{equation}
and for $\alpha=0$,
\begin{equation}\label{estimlambdazero}
 \lambda_{i_{\max}}\ges\exp\lt(-\Iz(E)\rt)  \qquad \textrm{and}\qquad \lambda_i\les \exp\lt(\Iz(E)\rt)\Phi^{N-2} V^{\frac{1}{N-1}} \qquad \text{ for } i\neq i_{\max},
\end{equation}
where the constants implicitly appearing in \eqref{estimlambdaalpha} and \eqref{estimlambdazero}  depend only on $N$ and $\alpha$.
\end{lemma}

\begin{proof}
Without loss of generality, we can assume that $\lambda_1\ge \lambda_2\ge \cdots \ge \lambda_N$. Then, since $V=\prod_{i=1}^N \lambda_i$ and $P(E)\lesssim \prod_{i=1}^{N-1} \lambda_i$, taking the ratio of these two quantities,
 we obtain $\lambda_N\ges V P(E)^{-1}=V^{\frac{1}{N-1}}\Phi^{-1}$. Now, since the $\lambda_i$ are decreasing (in particular $\lambda_i\ge \lambda_N$ for all $i$), this implies
 \[
 \Phi\ges V^{-\frac{N-2}{N-1}}\prod_{i=1}^{N-1} \lambda_i= V^{-\frac{N-2}{N-1}} \lambda_1\prod_{i=2}^{N-1} \lambda_i\ges   V^{-\frac{N-2}{N-1}} \lambda_1 V^{\frac{N-2}{N-1}}\Phi^{-(N-2)},
 \]
yielding \eqref{estimlambdageom}.\\
Assume now that $\alpha>0$. Then, from $\diam(E)\sim \lambda_1$, we get $\Ia(E)\ges \lambda_1^{-\alpha}$. If $N=2$, together with $\lambda_1\lambda_2=V$, this implies \eqref{estimlambdaalpha}. If $N\ge 3$, 
we infer as above that 
\[\Phi\ges V^{-\frac{N-2}{N-1}} \lambda_1 \lambda_2 \prod_{i=3}^{N-1}\lambda_i\ges V^{-\frac{N-2}{N-1}}\Ia(E)^{-\frac1\alpha} \lambda_2 V^{\frac{N-3}{N-1}}\Phi^{-(N-3)}\ges V^{-\frac{1}{N-1}}\Phi^{-(N-3)}\Ia(E)^{-\frac1\alpha} \lambda_2. \]
 This gives \eqref{estimlambdaalpha}. The case $\alpha=0$ follows analogously, using 
 the fact that $\Iz(E)\ge C-\log \lambda_1$.
 \end{proof}

The next result follows directly from John's lemma  \cite{john}.

\begin{lemma}\label{john}
There exists a dimensional constant $C_N>0$ such that for every convex body $E\subset \R^N$, up to a rotation and a translation, there exists $\mathcal{R}:=\prod_{i=1}^N [0,\lambda_i]$, 
 such that 
\[
\mathcal R\subseteq E\subseteq C_N \mathcal R.
\]
As a consequence $\diam(E)\sim \diam(\mathcal{R})$, $|E|\sim |\mathcal{R}|$, $P(E)\sim P(\mathcal{R})$ and $\Ia(E)\sim \Ia(\mathcal{R})$ for $\alpha>0$ (and $\exp(-\Iz(E))\sim\exp(-\Iz(\mathcal{R}))$).
\end{lemma}

With these two preliminary results at hand, we can prove existence of minimizers for \eqref{problema}.
\begin{theorem}\label{esistenza}
For every $Q>0$ and $\alpha\in[0,N)$,  \eqref{problema} has a minimizer.
\end{theorem}

\begin{proof}
 Let $E_n$  be a minimizing sequence and let us prove that $\diam(E_n)$ is uniformly bounded. Let $\mathcal{R}_n$ be the parallelepipeds given by Lemma \ref{john}. Since $\diam(E_n)\sim \diam(\mathcal{R}_n)$, it is enough to estimate $\diam(\mathcal{R}_n)$ from above.
 Let us begin with the  case $\alpha>0$. In this case, since $\Ia(\mathcal{R}_n)\ge 0$, by \eqref{estimlambdageom}, applied with $V=1$, we get
 \[\diam(\mathcal{R}_n)\les P(\mathcal{R}_n)^{N-1}\les \FaQ(E_n)^{N-1}.\]
In the case $\alpha=0$, from \eqref{estimlambdageom} and \eqref{estimlambdazero} applied to $V=1$, we get 
\[P(\mathcal{R}_n)\ges \exp\lt(-\frac{\Iz(\mathcal{R}_n)}{N-1}\rt)\]
so that 
\[\FzQ(\mathcal{R}_n)\ges \exp\lt(-\frac{\Iz(\mathcal{R}_n)}{N-1}\rt) + Q^2\,\Iz(\mathcal{R}_n),\]
from which we obtain that $|\Iz(\mathcal{R}_n)|$ is bounded and thus also $P(\mathcal{R}_n)$ is bounded, whence, arguing as above, we obtain a uniform bound on $\diam(\mathcal{R}_n)$. \\
Since the $E_n $'s are convex sets, up to a translation, we can extract a subsequence which converges in the  Hausdorff (and $L^1$) topology to some convex body $E$ of volume one.
Since the perimeter functional is lower semicontinuous with respect to the $L^1$ convergence, and the Riesz potential $\mathcal I_\alpha$ is lower semicontinuous with respect to the Hausdorff convergence (see \cite{landkof,saftot} and 
\cite[Proposition 2.2]{gnrI}),
we get that $E$ is a minimizer of \eqref{problema}.
\end{proof}

\section{Regularity of the planar charge distribution for the logarithmic potential}\label{secdens}
In this section we focus on the case $N=2$ and $\alpha=0$. Relying on classical results on harmonic measures, we show that for every convex set $E$, the corresponding optimal measure $\mu$ for $\Iz(E)$ is absolutely continuous with respect to $\Hu\restr \partial E$ with $L^p$ estimates.
Upon making that connection between $\mu$ and  harmonic measures, this fact is fairly classical. However, since we could not find a proper reference, we  recall (and slightly adapt) a
few useful results. Let us point out that most definitions and results of this section extend to the case $N\ge 3$ and $\alpha=N-2$, and to more general classes of sets.  In particular, for bounded Lipschitz sets, the fact that harmonic measures are absolutely continuous with respect to the surface 
measure with $L^p$ densities for $p>2$ was established in \cite{Dahl}, 
and extended later to more general domains   (see for instance \cite{KenTorfree,KenTor,JerKen}). The interest for harmonic measures stems from the fact that they bear a lot of geometric information  (see in particular \cite{AltCaff,KenTorfree}).
The main result of this section is the following.

\begin{theorem}\label{estimregmuconv}
 Let $E_n$ be a sequence of compact convex bodies converging to a convex body $E$ and let $\mu_n$ be the associated equilibrium measures. Then,  $\mu_n=f_n \Hu\LL \partial E_n$ and there exists $p>2$ and $M>0$
 (depending only on $E$) such that $f_n\in L^p(\partial E_n)$ with 
 \[\|f_n\|_{L^p(\partial E_n)}\le M.\]
 Moreover, if $E$ is smooth, then $p$ can be taken arbitrarily large.
\end{theorem}

\begin{remark}
	By applying the previous result with $E_n=E$, we get that that the equilibrium measure of a convex set is always in some $L^p(\partial E)$ with $p>2$.
	We stress also that the exponent $p$ and the bound on the $L^p$ norm of its equilibrium measure depend indeed on the set:  for instance, a sequence of convex sets with 
	smooth boundaries converging to a square cannot have equilibrium measures with densities uniformly bounded in $L^p$ for $p>4$.
\end{remark}

We will denote here $\Omega:=E^c$. Let us recall the definition of harmonic measures (see \cite{GarMar,KenTorfree}).
\begin{definition}
 Let $\Omega$ be a Lipschitz open set (bounded or unbounded) such that $\R^2\backslash \partial \Omega$ has two connected components, and let $X\in \Omega$, we denote by $G^X_\Omega$ the Green function of $\Omega$ with pole at $X$ i.e. the unique distributional solution of 
 \[
 -\Delta G_\Omega^X=\delta_X \quad \textrm{ in } \Omega \qquad \textrm{and} \qquad G_\Omega^X =0 \quad \textrm{ on } \partial \Omega,
 \]
 and by $\omega^X_\Omega$ the {\em harmonic measure of $\Omega$ with pole at $X$}, that is the unique (positive) measure such that for every $f\in C^0(\partial \Omega)$, the solution  $u$ of
 \[-\Delta u=0 \quad \textrm{ in } \Omega \qquad \textrm{and} \qquad u =f \quad \textrm{ on } \partial \Omega,\]
 satisfies
 \[u(X)=\int_{\partial \Omega} f(y) d\omega^X_\Omega(y).\]
 If $\Omega$ is unbounded with $\partial \Omega$ bounded and $0\in \overline{\Omega}^c$,  we call $\omega^\infty_\Omega$ the {\em harmonic measure of $\Omega$ with pole at infinity}, that is the unique probability
 measure on $\partial \Omega$ satisfying
 \[\int_{\partial \Omega}\phi d\omega^{\infty} =\int_{\Omega} u \Delta \phi \qquad \forall \phi\in C^{\infty}_c(\R^2)\]
 where $u$ is the solution  of
 \begin{equation}\label{Dirichletinf}\begin{cases}
    -\Delta u=0 & \textrm{in } \Omega\\
    u>0 & \textrm{in } \Omega\\
    u=0 &\textrm{on } \partial \Omega\\
    \lim_{|z|\to +\infty} \lt\{u(z)-\frac{1}{2\pi} \log |z|\rt\} \, \textrm{ exists and is finite}.
   \end{cases}
   \end{equation}
\end{definition}

\noindent
When it is clear from the context, we omit the dependence of $G^X,\omega^X$ or $\omega^\infty$ on the domain $\Omega$.
\begin{remark}
 For smooth domains, it is not hard to check that $\omega^X= \partial_\nu G^X \Hu\LL \partial \Omega$, and that  $\omega^{\infty}=\partial_\nu u \Hu\LL \partial \Omega$ where $\nu$ is the inward unit normal to $\Omega$. Moreover, for $\Omega$ unbounded, 
 if $h^\infty$ is the harmonic function in $\Omega$ with $h^\infty(z)=-\frac{1}{2\pi} \log |z|$ on $\partial \Omega$, then the function $u$ from \eqref{Dirichletinf} can also be defined by $u(z)=\frac{1}{2\pi} \log|z| +h^\infty(z)$.
\end{remark}
 
We may now make the connection between harmonic measures and equilibrium measures. For $E$ a Lipschitz bounded open set containing $0$, let $\mu$ be the optimal measure for $\Iz(E)$ and let 
\[v(x):=\int_{\partial E}  -\log(|x-y|) d\mu(y).\]
Since
\[
-\Delta v=2\pi \mu \textrm{ in } \R^2, \qquad v<\Iz(E) \textrm{ in } E^c \qquad \textrm{and } \qquad  v= \Iz(E) \textrm { on } \partial E, 
\]
 if we let $u:=(2\pi)^{-1}(\Iz(E)-v)$, we see that it satisfies \eqref{Dirichletinf} for $\Omega=E^c$. Therefore, $\mu=\omega^\infty_{E^c}$ (recall that $\mu(\partial E)=1$). 
For Lipschitz sets $\Omega$, it is well-known that $\omega^{\infty}$ is absolutely continuous with respect to $\Hu\LL \partial \Omega$ with density in $L^p(\partial \Omega)$ for some $p>1$ (see \cite[Theorem 4.2]{GarMar}). However,  we will need a stronger result, 
namely that it is in $L^p(\partial \Omega)$ for some $p>2$, with estimates on the $L^p$ norm depending only on the geometry of $\Omega$.


Given a convex body $E$ and a point $x\in\partial E$, we call {\it angle} of $\partial E$ at $x$
the angle spanned by the tangent cone $\cup_{\lambda>0}\lambda(E-x)$.

We now state a crucial lemma which relates in a quantitative way the regularity of $E$ with the integrability properties of the corresponding harmonic measure. This result is a slight  adaptation of \cite[Theorem  2]{WarSchob}.

\begin{lemma}\label{conformalregularity}
	Let $E$ be a convex body containing the origin in its interior, 
	let $\overline{\zeta}\in (0,\pi]$ be the minimal angle of $\partial E$,
	and let $p_c:=\frac{\pi}{\pi-\overline{\zeta}}+1$ if $\overline{\zeta}<\pi$
	and $p_c:=+\infty$ if $\overline{\zeta}=\pi$.
	Let also $E_n$ be a sequence of convex bodies converging to $E$ in the Hausdorff topology.
	Then, for  every $1\le p<p_c$, there exists $C(p,\partial E)$ such that 
	for $n$ large enough (depending on $p$), 
	every conformal map $\psi_n : E_n^c\to B_1$ with $\psi_n(\infty)=0$ satisfies
	\begin{equation}\label{est1}
	\int_{\partial E_n}|\psi'_n|^p\le C(p,\partial E),
	\end{equation} 
	where we indicate by $|\psi'_n|$ the absolute value of the derivative of $\psi_n$ (seen as a complex function).
	In particular, for $n$ large enough, $\psi'_n\in L^p(\partial E_n)$ for some $p>2$.  
\end{lemma} 

\begin{proof}
The scheme of  the proof follows that of   \cite[Theorem  2, Equation (9)]{WarSchob}, thus we limit ourselves to point out the main differences. We begin by noticing that although \cite[Theorem  2]{WarSchob} is written for bounded sets, up to composing with the map $z\to z^{-1}$ this does not create any difficulty.

We first introduce some notation from \cite{WarSchob}. Given a convex body $E$ we let 
 $\partial E=\{\gamma(s) \ : \ s\in[0,L]\}$ be an arclength parametrization of $\partial E$.
Notice that, for every $s$, the left and right derivatives
$\gamma'_{\pm}(s)$ exist and the angle $v(s)$ between $\gamma'(s)$ and a fixed direction, say  $e_1$, is a function of bounded variation. Up to changing the orientation of $\partial E$,
we can assume that $v$ is increasing.
%
We then let 
\begin{equation*}
\bar \eta:=\max_{s} [v(s^+)-v(s^-)]\ge 0,
\end{equation*}
where $v(s^\pm)$ are the left and right limits at $s$ of $v$. Notice that $\overline\zeta=\pi-\overline{\eta}$ is the minimal angle of $\partial E$.

Letting $\vphi_n:= \psi_n^{-1}$,
we want to prove that there exists $C(p,\partial E)$ such that 
\[
 \int_{\partial B_1} |\vphi_n'|^{-p}\le C(p,\partial E),
\]
for $n$ large enough and  for $p<\pi/\overline{\eta}$. 
By a change of variables, this yields \eqref{est1}. Let $p<p'<\pi/\overline{\eta}$, and let as in \cite{WarSchob},
\[ 
h:= \frac{1}{2\pi}(p \overline{\eta}+\pi)  \qquad \textrm{and} \qquad h':= \frac{1}{2\pi}(p'\overline{\eta}+\pi),
\]
so that 
\[\frac{\pi h}{p}>\frac{\pi h'}{p'}>\overline{\eta}.\]
Let now $v^n$ (respectively $v$) be the angle functions corresponding to the sets $E_n$ 
(respectively $E$).
As in \cite{WarSchob}, there exists $\delta>0$ such that for $s-s'\le \delta$,
\[v(s)-v(s')\le \frac{\pi h'}{p'}.\]
By the convexity of $E_n$ and by the convergence of $E_n$ to $E$, 
for $n$ large enough and for $s-s'\le \delta$ we get that 
\[v^n(s)-v^n(s')\le \frac{\pi h}{p}.\]
Let $L_n:=\Hu(\partial E_n)$ and let us extend $v^n$ to $\R$ by letting  for $s\ge 0$, 
$v^n(s):=v^n(L_n \lfloor s/ L_n\rfloor)+v^n(s-L_n \lfloor s/ L_n\rfloor)$, and similarly for $s\le 0$, 
so that $v^n$ is an increasing function with $(v^n)'$ periodic of period $L_n$.
Let now $k_n:=\lceil L_n/\delta \rceil\in\mathbb N$ and $\delta_n:=L/k_n$. By the convergence of $E_n$ to $E$, $k_n$ and $\delta_n$ are uniformly bounded from above and below. For $t\in[0,\delta_n]$, and $0\le j\le k_n$, let $s^t_j:= t+j \delta_n$.
Since 
\begin{align*}
 \int_0^{\delta_n}\sum_{j=0}^{k_n-1} \int_{s_j^t}^{s_{j+1}^t} \frac{v^n(s)-v^n(s^t_j)}{s-s^t_j} ds dt&=  \sum_{j=0}^{k_n-1} \int_0^{\delta_n} \int_0^{\delta_n} \frac{v^n(s+t+ j \delta_n)-v^n(t+j \delta_n)}{s} dt ds\\
 &= \int_0^{\delta_n} \frac{1}{s} \sum_{j=0}^{k_n-1} \int_0^{\delta_n}v^n(s+t+ j \delta_n)-v^n(t+j \delta_n) dt ds\\
 &= \int_0^{\delta_n} \frac{1}{s} \lt(\int_{L_n}^{L_n+s} v^n(t) dt -\int_0^s v^n(t) dt\rt) ds\\
 &\le 2\delta_n\sup_{[0,2L_n]} |v^n|\les \delta_n\|v\|_\infty,
\end{align*}
we can find $\overline{t}\in (0,\delta_n)$ such that 
\[
\sum_{j=0}^{k_n-1} \int_{s_j^{\overline{t}}}^{s_{j+1}^{\overline{t}}} \frac{v^n(s)-v^n(s^{\overline{t}}_j)}{s-s^{\overline{t}}_j} ds\les \|v\|_\infty.\]
For notational simplicity, let us simply denote $s_j:= s_{j}^{\overline{t}}$. Arguing as above, we can further assume that 
\[\sum_{j=0}^{k_n-1} \int_{s_j}^{s_{j+1}} \frac{v^n(s_{j+1})-v^n(s)}{s_{j+1}-s} ds\les \|v\|_\infty.\]
The proof then follows almost exactly as in \cite[Theorem  2]{WarSchob}, by replacing  the pointwise quantity
\[G^n_j:= \sup_{s_j<s<s_{j+1}} \frac{v^n(s)-v^n(s_j)}{s-s_j},\]
by the integral ones. There is just one additional change in the proof: letting 
$0\le \lambda_j^n:= v^n(s_{j+1})-v^n(s_j)\le \frac{\pi h}{p}$, we see that  in the estimates of \cite[Theorem  2]{WarSchob}, 
the quantity $\max_{\lambda_{j}^n\neq 0} 1/\lambda_j^n$ appears and could be unbounded in $n$. Let $\gamma_n(s)$ be the arclength parametrization of $\partial E_n$ and let $\theta_n(s)$ be such that 
$\gamma_n(s)=\vphi_n(e^{i\theta_n(s)})$. For $0<r<1$  and $j\in [0, k_n-1]$, if $\lambda_j^n\neq0$, we have 
\[\frac{1}{\lambda_j^n}\int_{s_j}^{s_{j+1}} dv^n(s)\int_{\theta_n(s_j)}^{\theta_n(s_{j+1})} \frac{dt}{|e^{i\theta_n(s)}-r e^{it}|^h}\les \frac{1}{1-h}.
\]
Using this estimate, the proof can be concluded  exactly as in \cite[Theorem  2]{WarSchob}.
\end{proof}




We can now prove Theorem \ref{estimregmuconv}.

\begin{proof}[Proof of Theorem \ref{estimregmuconv}]
 Without loss of generality we can assume that the sets $E_n$ and $E$ contain the origin in their interior.
 As observed above, we then have $\mu_n=\omega^{\infty}_{E_n^c}$.
 Let $\psi_n$ be a  conformal mapping from $E_n^c$ to $B_1$ with $\psi_n(\infty)=0$. 
 We have
 \[
 \mu_n = \omega^\infty_{E_n^c}= (\psi_n^{-1})_\sharp\, \omega^0_{B_1} 
 = (\psi_n^{-1})_\sharp\, \frac{\Hu \LL \partial B_1}{2\pi}
  = \frac{|\psi_n'|}{2\pi} \Hu \LL \partial E_n.
 \]
Then, Lemma \ref{conformalregularity}  gives the desired estimate.
\end{proof}

We will also need a similar estimate for $C^{1,\beta}$ sets.

\noindent
\begin{lemma}\label{regmu}
	Let  $E$ be a convex set with boundary of class $C^{1,\beta}$. Then, the optimal charge distribution $\mu$ is of class $C^{0,\beta}$ and in particular it is in $L^\infty(\partial E)$. Moreover, $\|\mu\|_{C^{0,\beta}}$ depends only on the $C^{1,\beta}$ norm of $\partial E$.
\end{lemma}
\begin{proof}
	Up to translation we can assume that $0\in E$ with ${\rm dist}(0,\partial E)\ge c$ (with $c$ depending only on the $C^{1,\beta}$ character of $\partial E$).
	By  \cite[Theorem  3.6]{pommerenke}, there exists a conformal mapping  $\psi$ of class $C^{1,\beta}$ which maps $E^c$ into $B_1$ 
	with $\psi(\infty)=0$ and $\|\psi\|_{C^{1,\beta}(E^c)}$ controlled by the $C^{1,\beta}$ character of $\partial E$.
	Since, as before, $\mu=(\psi^{-1})_\sharp \omega^0_{B_1}$, and the claim follows by Lemma \ref{conformalregularity}.
\end{proof}

\section{$C^{1,1}$-regularity of minimizers for $N=2$ and $\alpha=0$}\label{secreg}

In this section we show that any minimizer of \eqref{problema} has boundary of class $C^{1,1}$. We begin by showing that we can drop the volume constraint, by adding a volume penalization to the functional. This penalization is commonly used in isoperimetric type problems (see for instance \cite{EspFus,GolNo} and references therein).
Let $\Lambda$ be a positive number and define the functional
\[
\mathcal{G}_\Lambda (E):=P(E)+Q^2\mathcal I_0(E)+\Lambda\left||E|-1\right|.
\]
\begin{lemma}\label{lemrego2}
For every $Q_0>0$, there exists $\overline \Lambda>0$ such that, if $\Lambda>\overline{\Lambda}$ and $Q\le Q_0$, the minimizers of 
\begin{equation}\label{minLambda}
\min_{E\subseteq\R^2,\, \text{$E$ {\rm convex}}} \mathcal{G}_\Lambda (E)
\end{equation}
are also minimizers of \eqref{problema}
and vice-versa. Furthermore, the diameter of the minimizers of \eqref{minLambda} is uniformly bounded by a constant depending only on $Q_0$.
\end{lemma}

\begin{proof}
Let us fix $Q_0>0$ and let $Q<Q_0$. 
Let $B$ be a ball with $|B|=1$. Then for any $E\subset \R^2$ such that $\G(E)\le\G(B)$ we have
\[
\text{diam}(E)-Q^2 \log( \text{diam}(E)) \le \mathcal{G}_\Lambda(E)\le \mathcal{G}_\Lambda(B)=\FzQ(B)\les 1,
\]
where the constant involved depends only on $Q_0$.  For such sets,  $ \text{diam}(E)$ is bounded by a constant $R$ depending only on $Q_0$, and thus $I_0(E)\ge I_0(B_R)$. This implies that every minimizing sequence is uniformly bounded so that, up to passing to a subsequence, it converges in Hausdorff distance to a minimizer of $\G$ whose diameter is bounded by $R$.  Moreover, for 
\[
\Lambda>\overline\Lambda:=P(B)+Q_0^2\left(\Iz(B)+|\Iz(B_R)| \right)
\] 
we have that $|E|>0$. Indeed, for $|E|=0$  the inequality $\G(E)\le\G(B)$ implies $\Lambda\le\overline{\Lambda}$.

Notice that the minimum in \eqref{minLambda} is always less or equal than the minimum in \eqref{problema}. We are thus left to prove the opposite inequality. 
Assume that $E$ is not a minimizer for $\FzQ$. In this case we get that 
\[
\sigma:=||E|-1|>0.
\] 
\noindent

From the uniform bound on the diameter of $E$ we deduce that  $\Lambda \sigma$ is itself also bounded by a constant (again depending only on $Q_0$). 
From now on we assume that $|E|<1$, or equivalently, $|E|=1-\sigma$, 
since the other case is analogous.
Let us define
\[
F:=\frac{1}{(1-\sigma)^\frac12}E,
\] 
so that $|F|=1$. Then, by the minimality of $E$,  the homogeneity of the perimeter and recalling that 
\[
\mathcal I_0 (\lambda E)=\mathcal I_0 (E)-\log(\lambda),
\]
a Taylor expansion gives 
\[
\begin{aligned}
\Lambda\sigma &=\mathcal{G}_\Lambda(E)-\FzQ(E)\\
&\le \mathcal{G}_\Lambda(F)-\FzQ(E)\\
&=P(E)\left(1-\sigma\right)^{-\frac{1}{2}}+Q^2\mathcal I_0 (E)+\frac{1}{2}\log(1-\sigma) -\FzQ(E)\\
&\le P(E)(\left(1-\sigma\right)^{-\frac{1}{2}}-1)\\
&\le \frac{P(E)}{2} \sigma,
\end{aligned}
\]
so that $\Lambda\le \frac{P(E)}{2}\les1$. Therefore, if $\Lambda$ is large enough,  we must have $\sigma=0$ or equivalently that $E$ is also a minimizer of $\FzQ$.

%
\end{proof}

 Let  now $E$ be a minimizer of \eqref{minLambda}. In order to prove the regularity of $E$, we 
 shall construct a competitor in the following way: since $E$ is a convex body, 
 there exists $\eps_0$ such that for $\eps\le \eps_0$,  and every  $x_0\in \partial E$, we have 
 $\partial E\cap \partial B_\eps(x_0)=\{\xu,\xd\}$ (in particular $|x_0-x^\eps_i|=\eps$). 
 Let us fix  $x_0$.  For $\eps\le \eps_0$, let $\xu$, $\xd$ be given as above and let $L_\eps$ be the line joining $\xu$ to $\xd$. Denote by $H_\eps^+$  the half space with boundary $L_\eps$ containing $x_0$ (and $H_\eps^-$ be its complementary). We then define our competitor as
\[E_\eps:= E\cap H^-_\eps.\]
Let us fix some further notation (see Figure \ref{fig1}):
	\begin{itemize}
		\item[-] We denote by $\Pi:\partial E\cap H_\eps^+\to L_\eps$ the projection 
		of the cap of $\partial E$ inside $H_\varepsilon^+$, on $L_\eps$. We shall extend $\Pi$ to the whole $\partial E$ as the identity, outside $\partial E\cap H_\eps^+$.
		\item[-] If $f \H^1\restr \partial E$ is the optimal measure for $\Iz(E)$, we let 
		$f_\eps:= \Pi_\sharp f$ (which is defined on $\partial E_\eps$) so that  $\mu_\eps:= f_\eps H^1 \restr \partial E_\eps$ is a competitor for $\Iz(E_\eps)$. 
		\item[-]  For $x,y\in\partial E$, we denote by $\gamma_\eps(x,y)$ the acute angle between the line $L_{x,y}$ joining   $x$ to $y$ and $L_\eps$  (if $L_{x,y}$ is parallel to $L_\eps$, we set $\gamma_\eps(x,y)=0$).
		\item[-] If $y=x_0$, then we denote $\gamma_\eps(x):=\gamma_\eps(x,x_0)$.
		\item[-] We let $\gamma_\eps:=\gamma_\eps(\xu)=\gamma_\eps(\xd)$. 
		\item[-] We let $\partial B_{3\eps}(x_0)\cap \partial E=\{x_1^{3\varepsilon},x_2^{3\varepsilon}\}$.
		As before, we define $H_{3\eps}^+$ as the half space bounded by $L_{x_1^{3\eps},x_2^{3\eps}}$ containing $x_0$ and $H_{3\eps}^-$ its complementary. We then let $\Sigma_\eps:=\partial E\cap H^+_\eps$, $\Sigma_{3\eps}:= \partial E\cap H_{3\eps}^+$ and $\Gamma_\eps:= \partial E\cap H_{3\eps}^-$. 
		\item[-] We let $\Delta V:= |E|-|E_\eps|$, $\Delta P:= P(E)-P(E_\eps)$ and $\Delta \Iz:= \Iz(E_\eps)-\Iz(E)$.
	\end{itemize}
\begin{figure}
	\includegraphics[scale=1.5]{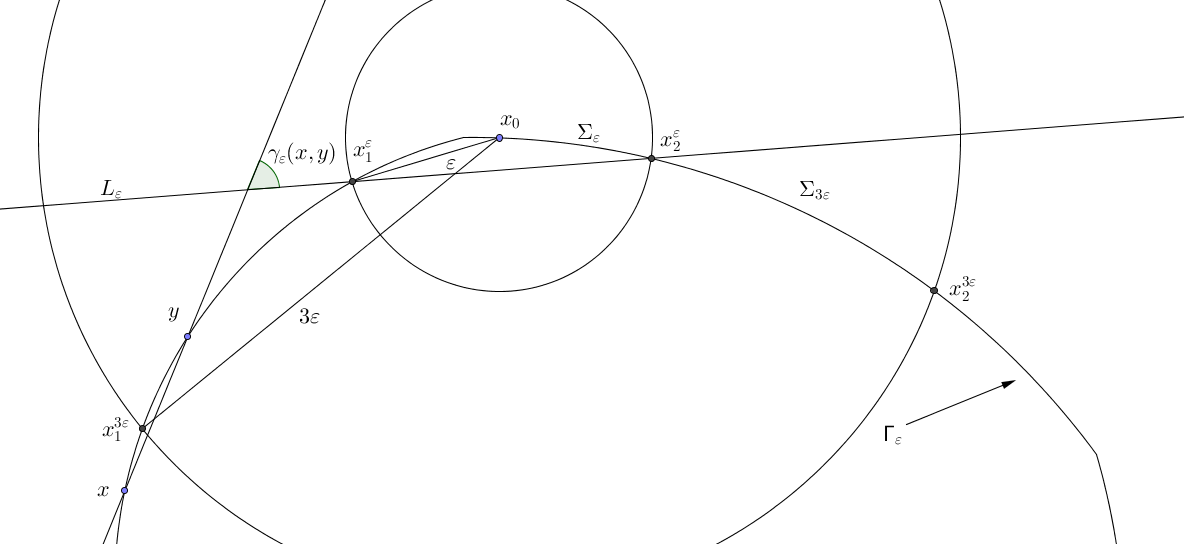}
	\caption{}\label{fig1}
\end{figure}
		
			We point out some simple remarks:
			\begin{itemize}
				\item[-] Thanks to Theorem \ref{estimregmuconv}  we have that the optimal measure $f$ satisfies $f\in L^p(\partial E)$ for some $p=p(E)>2$.
				\item[-] If $E$ is a convex body then $\gamma_\eps$ is bounded away from $\frac{\pi}{2}$ and $|x_1^{3\eps}-\xu|\sim |x_2^{3\eps}-\xd|\sim\eps$.
				\item[-] The quantities $\Delta V$, $\Delta P$ and $\Delta \Iz$ are nonnegative by definition.
				\item[-] All the constants involved up to now depend only on the Lipschitz character of $\partial E$. In particular, if $E_n$ is a sequence of convex bodies
				converging to a convex body $E$, then these constants depend only on the geometry of $E$. 
			\end{itemize}
		
		 
	Before stating the main result of this section, we prove two  regularity lemmas.
	\begin{lemma}\label{geomreglemma}
	Let $0<\beta\le 1$ and $C,\eps_0>0$ be given. Then, every convex body $E$ such that for every $x_0\in \partial E$ and every $\eps\le \eps_0$,
	\begin{equation}\label{hypdeltaV}
	\Delta V\le C \eps^{2+\beta}
	\end{equation}
	is $C^{1,\beta}$ with $C^{1,\beta}-$norm depending only on the Lipschitz character of $\partial E$, $\eps_0$ and $C$. 
	\end{lemma}
	\begin{proof}
	Let $x_0\in \partial E$ be fixed.
	Since $E$ is convex, there exist $R>0$ and  a convex function $u:I\to\R$ such that $\partial E\cap B_R(x_0)=\{(t,u(t))\ : \ t\in I\}$ for some interval $I\subset \R$. Furthermore, 	$\|u'\|_{L^\infty}\les 1$. Let $\bar x\in \partial E \cap B_R(x_0)$. 
	Without loss of generality, we can assume that  $\bar x=0=(0,u(0))$. By convexity of $u$, up to adding a linear function, we can further assume that $u\ge 0$  in $I$.
Thanks to the Lipschitz bound on $u$, for $x=(t,u(t))\in \partial E\cap B_R(x_0)$, we  have
		\begin{equation}\label{tequieps}
		|x|=(t^2+|u(t)|^2)^{1/2}\sim t.
		\end{equation}
		Let now $\eps>0$.  For $\delta>0$, let $-1 \ll t_1^\delta<0<t_2^\delta\ll1$ such that $x_i^\delta=(t_i^\delta, u(t_i^\delta))$ 
		for $i=1,2$ (see the notation above). By \eqref{tequieps}, there exists  $\lambda>0$ depending only on the Lipschitz character of $u$, such that $|t_i^{\lambda \eps}|\ge \eps$. Without loss of generality, we can now assume that $u(-\eps)\le u(\eps)$.
			    In particular, considering the $\Delta V$ associated to $\lambda \eps$, we have that (see Figure \ref{fig2})
			    \[ 
			    \begin{aligned}
			    \Delta V&\ge  2\varepsilon u(\varepsilon)-\frac{2\varepsilon(u(\varepsilon)-u(-\varepsilon))}{2}-\int_{-\varepsilon}^{\varepsilon}u(t)\,dt\\
			    &=\varepsilon(u(\varepsilon)+u(-\varepsilon))
			    -\int_{-\varepsilon}^{\varepsilon}u(t)\,dt\,.
			    \end{aligned}
			    \]
\begin{figure}
		\includegraphics[scale=1]{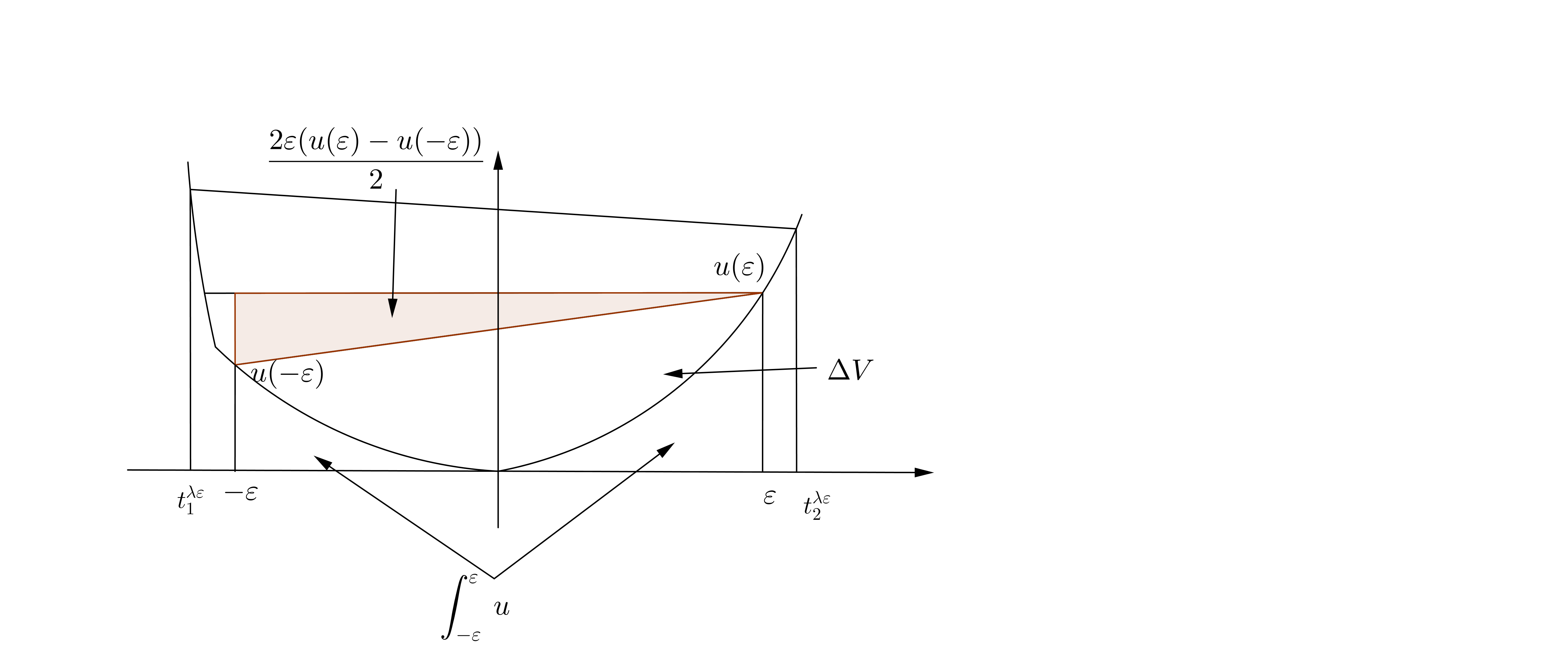}
		\caption{}\label{fig2}
\end{figure}

	
	\noindent	Since $u$ is decreasing in $[-\eps,0]$ and increasing in $[0,\eps]$, this means that both 
	\begin{equation}\label{bothhold}
	 \eps u(\eps)-\int_0^\eps u\les \eps^{2+\beta} \qquad \textrm{and } \qquad \eps u(-\eps)-\int^0_{-\eps} u\les \eps^{2+\beta}
	\end{equation}
	hold. Let us prove that this implies that for $|t|$ small enough
	\begin{equation}\label{estimu}
	 u(t)\les |t|^{1+\beta}.
	\end{equation}
We can assume without loss of generality that $t>0$. By \eqref{bothhold} and monotonicity of $u$,
\[
 t u(t)\le C t^{2+\beta}+ \int_0^{t/2} u +\int_{t/2}^t u\le C t^{2+\beta}+\frac{t}{2}(u(t/2)+u(t))
\]
from which we obtain
\[
 u(t)-u(t/2)\les t^{1+\beta}.
\]
Applying this for  $k\ge 0$ to $t_k= 2^{-k} t$ and summing over $k$ we obtain 
\[
 u(t)\les \sum_{k=0}^{\infty} \lt(2^{-k} t\rt)^{1+\beta}\les t^{1+\beta},
\]
that is \eqref{estimu}.

		In other words, we have proven that $u$ is differentiable in zero with $u'(0)=0$ and that  for $|t|$ small enough,
		\[|u(t)-u(0)-u'(0) t|\les |t|^{1+\beta}.\]
		Since the point zero was arbitrarily chosen, this yields that $u$ is differentiable everywhere and that  for $t,s\in I$ with $|t-s|$ small enough, 
		\[|u(t)-u(s)-u'(s)(t-s)|\les |t-s|^{\beta+1},\]
		which is equivalent to the $C^{1,\beta}$ regularity of $\partial E$\footnote{indeed, for $|s-t|\le \eps_1$, $|u'(t)-u'(s)|\le |t-s|^{-1}(|u(t)-u(s)-u'(s)(t-s)|+|u(s)-u(t)-u'(t)(s-t)|) \les |t-s|^{\beta}$}.
	\end{proof}

	\begin{lemma}\label{geomlemma2}
	Suppose that the minimizer  $E$ for \eqref{minLambda} has boundary of class $C^{1,\beta}$, for some $0<\beta<1$. Then, there exists $R>0$ (depending only on the $C^{1,\beta}$ character of $\partial E$)
	such that for every $x_0\in \partial E$, $x\in \Sigma_\eps$ and $y\in B_R(x_0)$,
	\begin{equation}\label{estimgammaxy}
	\gamma_\eps(x,y)\les \eps^\beta +|x-y|^\beta.
	\end{equation}
	\end{lemma}
	\begin{proof}

		\noindent 		Without loss of generality, we can assume that $x_0=0$. As in the proof of Lemma \ref{geomreglemma}, since $E$ is convex and of class $C^{1,\beta}$, in the ball $B_R(0)$, 
		for a small enough $R$, $\partial E$ is a graph over its tangent of a $C^{1,\beta}$ function $u$. Up to a rotation, we can further assume that this tangent is horizontal so
		that for some interval $I\subset \R$, we have  $\partial E\cap B_R(0)=\{ (t,u(t))\ : \ t\in I\}$. In particular, if $x=(t,u(t))\in \partial E\cap B_R(0)$, $|u(t)|\les |t|^{1+\beta}$ and $|u'(t)|\les |t|^\beta$. \\
\noindent
		For $x=(t,u(t))\in \Sigma_\eps$ and $y=(s,u(s))\in B_R(0)$, let $\tilde{\gamma}_\eps(x,y)$ be the angle between $L_{x,y}$ and the horizontal line, i.e., $\tan(\tilde{\gamma}_\eps(x,y))=\frac{|u(t)-u(s)|}{|t-s|}$. Let us begin by estimating
		$\tilde{\gamma}_\eps$. First, if $|x-y|\les \eps$ (which thanks to \eqref{tequieps} amounts to $|t-s|\les \eps$ and thus since $x\in \Sigma_\eps$, $|t|+|s|\les \eps$),
		\[\tilde{\gamma}_\eps(x,y)\sim\frac{|u(t)-u(s)|}{|t-s|}\le \sup_{r\in[s,t]} |u'(r)|\les \eps^\beta.\]
		Otherwise, if $|x-y|\gg \eps$, since $|x|\les \eps$, we have $|x-y|\sim |y|\sim |s|$ and thus 
		\[\tilde{\gamma}_\eps(x,y)\les\frac{|u(t)|+|u(s)|}{|t-s|}\les \frac{\eps^{1+\beta}+ |s|^{1+\beta}}{|s|}\les |s|^{\beta}\les |x-y|^{\beta}.\]
		Putting these estimates together, we find
		\begin{equation}\label{estimgammatilde}
		\tilde{\gamma}_\eps(x,y)\les \eps^{\beta}+|x-y|^{\beta}.
		\end{equation}
		Let $\xi_\eps$ be the angle between $L_\eps$ and the horizontal line (see Figure \ref{fig5}). Since $\gamma_\eps(x,y)=\tilde{\gamma}_\eps\pm \xi_\eps$, \eqref{estimgammaxy} 
		holds provided that we can show
		\begin{equation}\label{betaeps}
		\xi_\eps\les \eps^{\beta}.
		\end{equation}
		Let $t_1^\eps,t_2^\eps\sim\eps$ be such that $\xu=(-t_1^\eps,u(-t_1^\eps))$ and $\xd=(t_2^\eps,u(t_2^\eps))$. We see that $\xi_\eps$ is maximal if $u(-t_1^\eps)=0$,
		and then $t_1^\eps=\eps$. In that case, $\tan \xi_\eps= \frac{u(t_2^\eps)}{\eps+t_2^\eps}$.
		\begin{figure}
			\includegraphics[scale=.75]{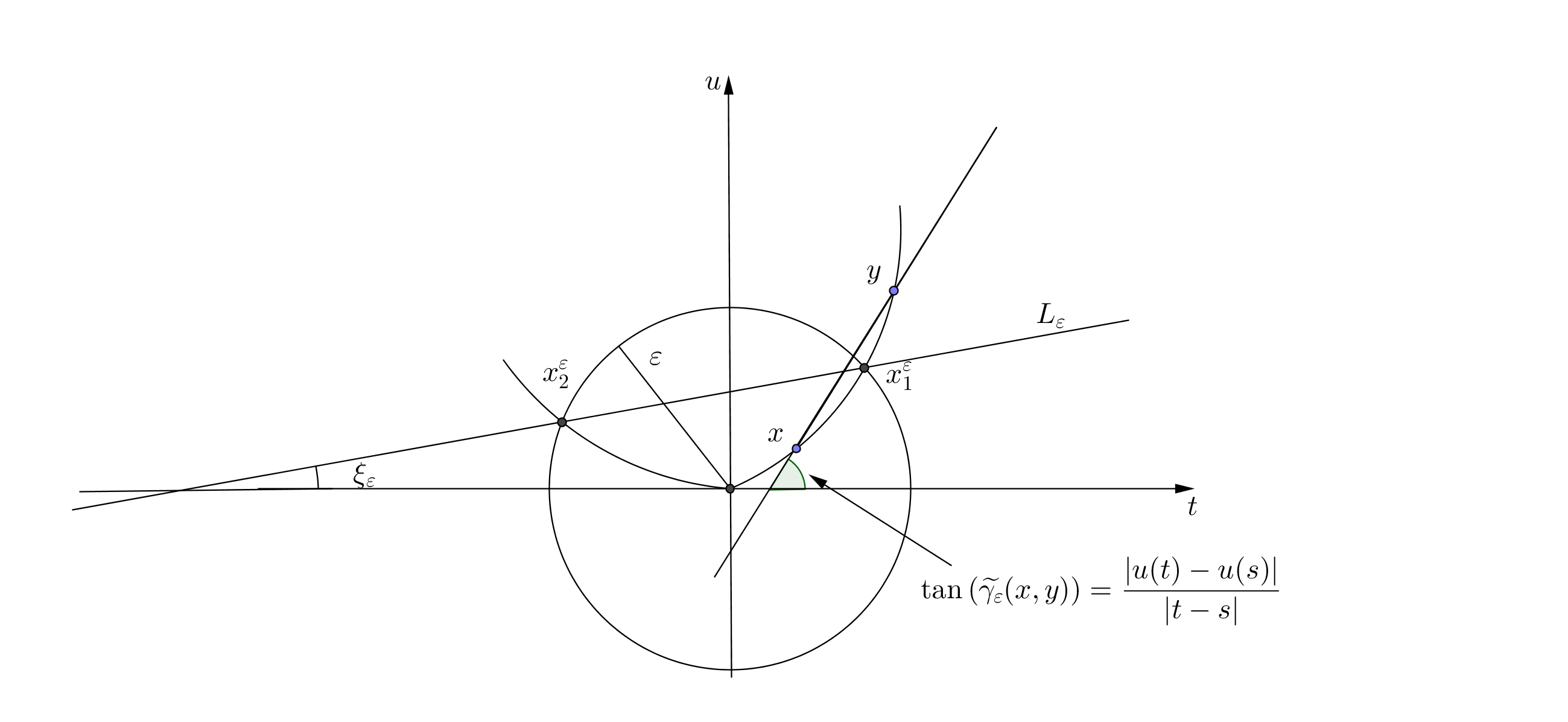}
			\caption{}\label{fig5}
			\end{figure}
			Since $u(t_2^\eps)\les \eps^{1+\beta}$, and $t^2_\eps\les \eps$, we obtain 
			\[\xi_\eps\sim \tan \xi_\eps \les \frac{\eps^{1+\beta}}{\eps}= \eps^\beta,\]
			proving \eqref{betaeps}. This concludes the proof of \eqref{estimgammaxy}.
\end{proof}
We pass now to the main result of this section.	
	\begin{theorem} \label{mainregTh}
		Every  minimizer of \eqref{minLambda} is  $C^{1,1}$. Moreover, for every $Q_0$ and every $Q\le Q_0$, the $C^{1,1}$ character of $\partial E$ depends only on $Q_0$, the Lipschitz character of $\partial E$ and $\|f\|_{L^p(\partial E)}$. 
	\end{theorem}
	\begin{proof}
Let $E$ be a minimizer of \eqref{minLambda}, $x_0\in \partial E$ be fixed and let $\eps\le \eps_0$. With the above notation in force, we begin by observing that using $E_\eps$ as a competitor, by minimality of $E$ for \eqref{minLambda}, we have
\begin{equation}\label{mainestimreg}
Q^2 \Delta \Iz\ge \Delta P-\Lambda \Delta V.
\end{equation}
We are thus going to estimate $\Delta \Iz$,  $\Delta P$ and  $\Delta V$ in terms of $\eps$ and $\gamma_\eps$. This will give us  a quantitative decay estimate for  $\gamma_\eps$. This in turn, in light of \eqref{DeltaV} below and Lemma \ref{geomreglemma}, will provide the desired regularity of $E$.
	 
	 \medskip
	
				{\it Step 1 (Volume estimate)}: In this first step, we prove that  
				
				\begin{equation}\label{DeltaV}
				\Delta V\sim \eps^2 \gamma_\eps \,. 
				\end{equation}			
		
\noindent		By construction, we have $\Delta V=|E|-|E_\eps|=|E\cap H_\eps^+|$.  By convexity, we first have that  the triangle with vertices $x_0,\xu,\xd$ is contained inside $E\cap H_\eps^+$. By convexity again, letting  ${\bar x_1}^\eps$ be the point of $\partial B_\eps(x_0)$ diametrically opposed to $\xu$  (and similarly for ${\bar x_2}^\eps$),  
	we get that $E\cap H_\eps^+$ is contained in the union of the triangles of vertices $x_2^\varepsilon,x_1^\varepsilon, {\bar x_1}^\varepsilon$ and $x_1^\varepsilon,x_1^\varepsilon, {\bar x_2}^\varepsilon$ (see Figure \ref{fig3}).
\begin{figure}
\includegraphics[scale=1]{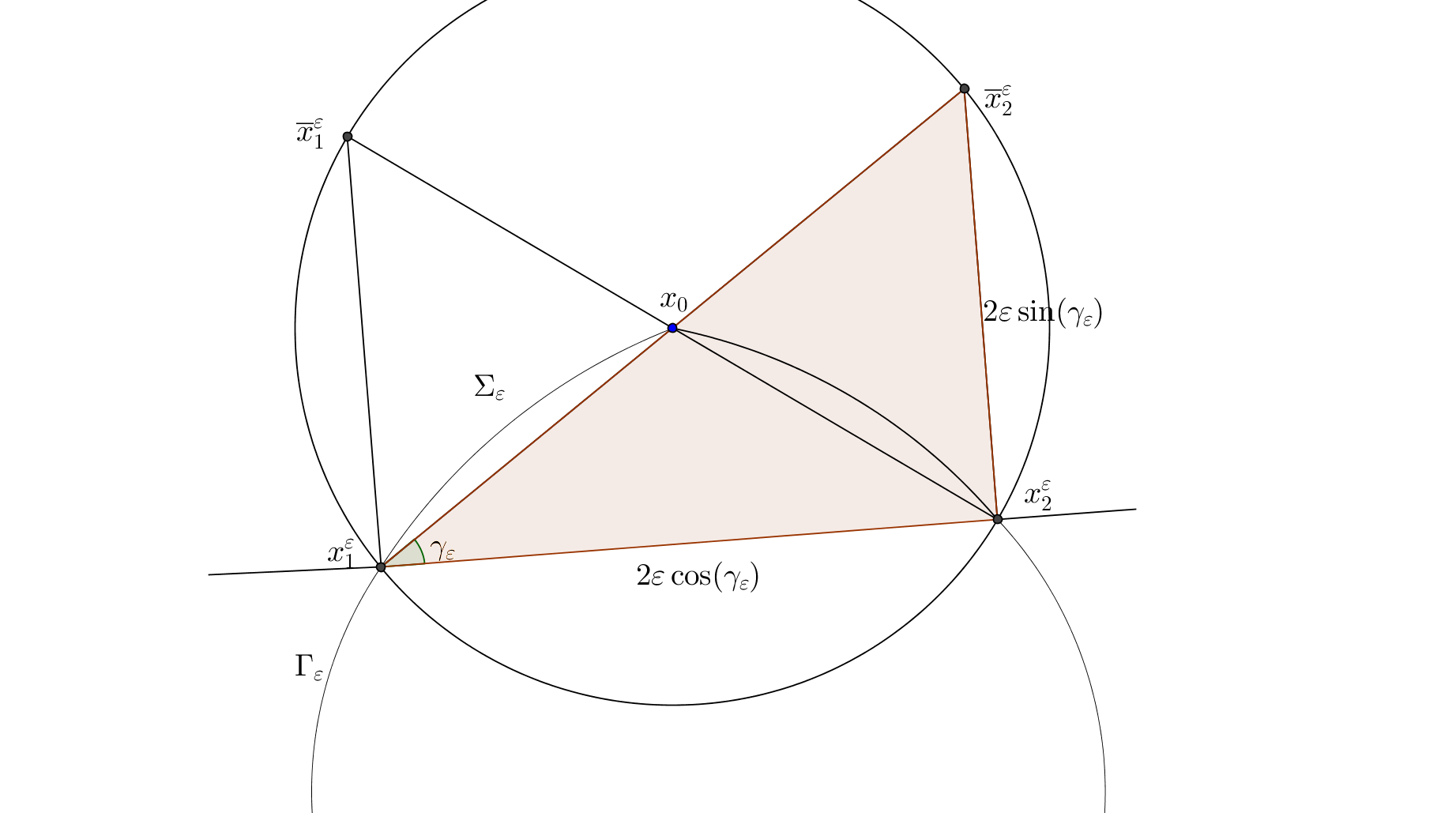}
	\caption{$\Delta V$ is contained in the union of the triangles of vertices $x_1^\varepsilon,x_2^\varepsilon, {\bar x_1}^\varepsilon$ and $x_1^\varepsilon,x_2^\varepsilon, {\bar x_2}^\varepsilon$.}\label{fig3}
\end{figure}			
		
	
\noindent
		Therefore, we obtain
		\[
		\Delta V\sim\ \eps^2 \cos \gamma_\eps \sin \gamma_\eps \sim  \varepsilon^2\gamma_\eps.
		\]

	\medskip
	
	{\it	Step 2 (Perimeter estimate):} Since the triangle with vertices $x_0,\xu,\xd$ is contained inside $E\cap H_\eps^+$, it holds  
		\begin{equation}\label{DeltaP}
		\Delta P=P(E)-P(E_\varepsilon)\ge 2\varepsilon\left(1-\cos \gamma_\eps\right)\ges \varepsilon\gamma_\eps^2.
		\end{equation}
		\medskip
		
	{	\it	Step 3 (Non-local energy estimate):}	We now estimate $\Delta \Iz$. Since $\mu_\eps$ is a competitor for $\Iz(E_\varepsilon)$,  recalling that $\Pi$ is the identity outside $\Sigma_\eps$, we have
		\[
		\begin{aligned} 
		\Delta\Iz&=\Iz(E_\varepsilon)-\Iz(E)\\
		&\le \int_{\partial E_\varepsilon\times\partial E_\varepsilon}f_\varepsilon(x)f_\varepsilon(y)\log\left(\frac{1}{|x-y|} \right)-\int_{\partial E\times\partial E}f(x)f(y)\log\left(\frac{1}{|x-y|} \right)\\
		&=\int_{\partial E\times\partial E}f(x)f(y)\log\left(\frac{1}{|\Pi(x)-\Pi(y)|} \right)-\int_{\partial E\times\partial E}f(x)f(y)\log\left(\frac{1}{|x-y|} \right)\\
		&= \int_{\partial E\times\partial E}f(x)f(y)\log\left(\frac{|x-y|}{|\Pi(x)-\Pi(y)|} \right).
		\end{aligned}
		\]
		Since for $x,y\in \Sigma_{\eps}^c$, $|\Pi(x)-\Pi(y)|=|x-y|$, 
		\[
		\begin{aligned}
		\Delta\Iz&\le  \int_{{\Sigma_{3\varepsilon}}\times{\Sigma_{3\varepsilon}}}f(x)f(y)\log\left(\frac{|x-y|}{|\Pi(x)-\Pi(y)|}\right) \\
		&+2\int_{\Sigma_\varepsilon}\int_{\Gamma_\varepsilon}f(x)f(y)\log\left(\frac{|x-y|}{|\Pi(x)-y|}\right) \\
		&=:I_1+2I_2.
		\end{aligned}
		\]
		We first estimate $I_1$:
		\[
		\begin{aligned}
		I_1&=\int_{{\Sigma_{3\varepsilon}}\times{\Sigma_{3\varepsilon}}}f(x)f(y)\log\left(1+\frac{|x-y|-|\Pi(x)-\Pi(y)|}{|\Pi(x)-\Pi(y)|}\right)\\
		&\le \int_{{\Sigma_{3\varepsilon}}\times{\Sigma_{3\varepsilon}}}f(x)f(y)\frac{|x-y|-|\Pi(x)-\Pi(y)|}{|\Pi(x)-\Pi(y)|}.
		\end{aligned}
		\]
		Since for any $x,y\in{\Sigma_{3\varepsilon}}$ we have (with equality if $x,y\in \Sigma_\eps$),
		\[
		\cos(\gamma_\eps(x,y))|x-y|\le |\Pi(x)-\Pi(y)|,
		\]
		we get 
		
		\begin{equation}\label{I1prep}
		I_1\le\int_{{\Sigma_{3\varepsilon}}\times{\Sigma_{3\varepsilon}}}f(x)f(y) \left(\frac{1}{\cos(\gamma_\eps(x,y))}-1\right)\les \int_{{\Sigma_{3\varepsilon}}\times{\Sigma_{3\varepsilon}}} \gamma_\eps^2(x,y) f(x)f(y). 
		\end{equation}
	Using then H\"older's inequality (recall that $f\in L^p(\partial E)$ for some  $p>2$) to get
		\begin{equation}\label{Holderf}
		\int_{{\Sigma_{3\varepsilon}}} f\le \lt(\int_{{\Sigma_{3\varepsilon}}} f^p\rt)^{1/p} \H^1({\Sigma_{3\varepsilon}})^{\frac{p-1}{p}}\les \eps^{ \frac{p-1}{p}},
		\end{equation}
		and $\gamma_\eps(x,y)\les 1$, we obtain
		
		\begin{equation}\label{estimI1}
		I_1\les \eps^{2\frac{p-1}{p}}. 
		\end{equation}
\noindent
		We can now estimate $I_2$:
		\[
		\begin{aligned}
		I_2&=\int_{\Sigma_\varepsilon}\int_{\Gamma_\varepsilon}f(x)f(y)\log\left(1+\left(\frac{|x-y|-|\Pi(x)-y|}{|\Pi(x)-y|}\right)\right)\\
		&\le\int_{\Sigma_\varepsilon}\int_{\Gamma_\varepsilon}f(x)f(y)\left(\frac{|x-y|-|\Pi(x)-y|}{|\Pi(x)-y|}\right).
		\end{aligned}
		\]
Denote by $z$ the projection of $\Pi(x)$ on the line containing $x$ and $y$. Then, since the projection is a $1$-Lipschitz function, it holds $|z-y|\le |\Pi(x)-y|$. Thus,
		\[|x-y|-|y-\Pi(x)|=|x-z|+ |z-y|-|y-\Pi(x)|\le |x-z|.\]
	Arguing as in {\it Step 1}, we get $|x-\Pi(x)|\le |\overline{x}_2^\eps-x_2^\eps|\les \eps\gamma_\eps$. Furthermore, the angle $\widehat{z\Pi(x)x}$ equals $\gamma_\eps(x,y)$ (see Figure \ref{fig4}), so that 
		\begin{figure}
			\includegraphics[scale=0.6]{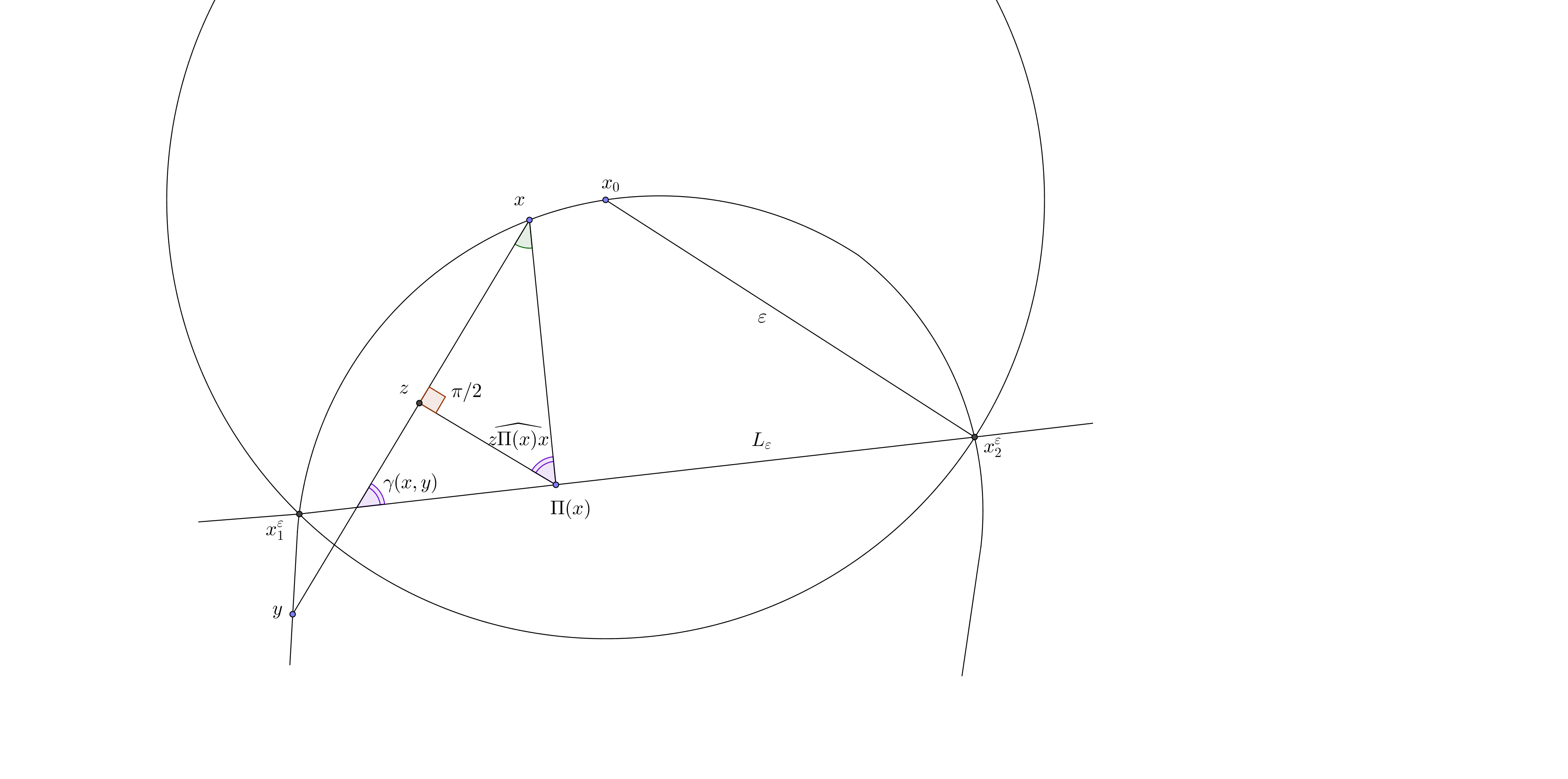}
			\caption{The angle $\widehat{z\Pi(x)x}$ equals $\gamma(x,y)$.}\label{fig4}
		\end{figure}
		\[
		|x-y|-|y-\Pi(x)|\le |x-z|=|x-\Pi(x)|\sin(\gamma_\eps(x,y))\lesssim \varepsilon \gamma_\eps\gamma_\eps(x,y). 
		\] 
		On the other hand, since $|y-x|\ge 2\eps$ (indeed $|x-x_0|\le\varepsilon$ and $|y-x_0|\ge3\varepsilon$), we have 
		\[
		 |y-\Pi(x)|\ge |y-x|-|x-\Pi(x)|\ges |y-x|-\eps \ges |y-x|.
		\]

\noindent		 Therefore,
		\begin{equation}\label{estimI2}
		I_2\lesssim \varepsilon\gamma_\eps \int_{\Sigma_\varepsilon}\int_{\Gamma_\varepsilon}  \frac{f(x)f(y)\gamma_\eps(x,y)}{|y- x|}.
		\end{equation}
		There exists $M>0$ which depends only on the Lipschitz character of $\partial E$ such that for $x\in \Sigma_\eps$ and  $y\in \Gamma_\eps\cap B_M(x_0)$,
		\[|y- x|\ge \min_{i=1,2} |y-{x}_i^\eps|.\]

\noindent		Let $\Gamma_\eps^N:=\Gamma_\eps\cap B_M(x_0)$ and $\Gamma_\eps^F:=\Gamma_\eps\cap B^c_M(x_0)$. We then have 
		\begin{align*}I_2&\lesssim \varepsilon\gamma_\eps \lt( \int_{\Sigma_\varepsilon\times \Gamma_\eps^N}  \frac{f(x)f(y)\gamma_\eps(x,y)}{\min_{i}|y- {x}_i^\eps|} + \int_{\Sigma_\varepsilon\times \Gamma_\eps^F}  f(x)f(y)\gamma_\eps(x,y)\rt)\\
		&=:I_2^N+I_2^F.
		\end{align*}
		We begin by estimating $I_2^F$. Since $\gamma_\eps(x,y)\les 1$, using H\"older's inequality we find
			\begin{equation}\label{estimIF}
			\begin{aligned} 
			I_2^F&\les\varepsilon\gamma_\varepsilon\left(\int_{\Gamma_\varepsilon}f\right)\left(\int_{\Sigma_\varepsilon}f\right)\\
			&\le \varepsilon\gamma_\varepsilon\|f\|_{L^p}\H^1(\Gamma_\varepsilon)^{1-\frac1p}\|f\|_{L^p}\H^1(\Sigma_\varepsilon)^{1-\frac1p}\\
			&\les \varepsilon\gamma_\varepsilon\H^1(\Sigma_\varepsilon)^{1-\frac1p}\\
			&\les \varepsilon^{2-\frac1p} \gamma_\eps.			
			\end{aligned} 
			\end{equation}
			
We can now estimate $I_2^N$. Recall that
\begin{equation}\label{definIN}
I_2^N:=\varepsilon\gamma_\eps\int_{\Sigma_\varepsilon\times \Gamma_\eps^N}  \frac{f(x)f(y)\gamma_\eps(x,y)}{\min_{i}|y- {x}_i^\eps|}. 
\end{equation}
As before, we use $\gamma_\eps(x,y)\les 1$ together with H\"older's inequality applied twice to get 
\[
\int_{\Sigma_\varepsilon\times \Gamma_\eps^N}  \frac{f(x)f(y)\gamma_\eps(x,y)}{\min_{i}|y- {x}_i^\eps|}\les \eps^{1-1/p}\lt(\int_{\Gamma_\eps^N} \frac1{\min_{i}|y- {x}_i^\eps|^{p/(p-1)}}\rt)^{(p-1)/p}.
\] 
Since $E$ is convex, its boundary can be locally parametrized by Lipschitz functions  so that, if $M$ is small enough (depending only on the Lipschitz regularity of $\partial E$), then for $y\in \Gamma_\eps^N$, $\min_i \ell(y,\tilde{x}_i^\eps) \sim \min_{i}|y- \tilde{x}_i^\eps|$ (where $\ell(x,y)$ denotes the geodesic distance  on $\partial E$). From this we get
\[\int_{\Gamma_\eps^N} \frac1{\min_{i}|y- {x}_i^\eps|^{p/(p-1)}}\les \eps^{-1/(p-1)}.\]
From this we conclude that 
\begin{equation}\label{estimIN}
I_2^N\les \gamma_\eps \eps^{2-\frac{2}{p}}.
\end{equation}
\medskip

					{\it Step 4 ($C^{1,\beta}$ regularity):} We now prove that $E$ has boundary of class $C^{1,\beta}$. To this aim, we can  assume that $\Delta V\ll \Delta P$. Indeed, if $\Delta V\ges \Delta P$, thanks to \eqref{DeltaV} and \eqref{DeltaP}, we would get $\gamma_\eps\les \eps$ and thus $\Delta V\les \eps^3$, which by Lemma \ref{geomreglemma} would  already ensure the $C^{1,1}$ regularity of $\partial E$. Using \eqref{mainestimreg}, \eqref{DeltaP}, \eqref{estimI1}, \eqref{estimIF} and \eqref{estimIN}, we get
					\begin{equation}\label{formulastep4}
					Q^2 (\eps^{1- \frac{2}{p}}+\gamma_\eps (\eps^{1-\frac{1}{p}}+\eps^{1-\frac{2}{p}}))\ges \gamma_\eps^2.
					\end{equation}
					Now since $\eps^{1-\frac{1}{p}}\les \eps^{1-\frac{2}{p}}$, this reduces further to 
					\begin{equation}\label{formulastep4bis}
					Q^2 (\eps^{1- \frac{2}{p}}+\gamma_\eps \eps^{1-\frac{2}{p}})\ges \gamma_\eps^2.
					\end{equation}
					We can now distinguish two cases. Either $Q^2 \eps^{2(\frac{1}{2}- \frac{1}{p})}\ges \gamma_\eps^2$ and then
					$\gamma_\eps\les Q\eps^{(\frac12- \frac{1}{p})}$ or $Q^2\gamma_\eps \eps^{1-\frac{2}{p}}\ges \gamma_\eps^2$ and then
					$\gamma_\eps\les Q^2 \eps^{1-\frac{2}{p}}$. Thus in both cases, since $p>2$, we find $\gamma_\eps\les Q\varepsilon^\beta$ for some $\beta>0$ and we can conclude, by means of \eqref{DeltaV} and Lemma \ref{geomreglemma}, that $\partial E$ is $C^{1,\beta}$.

%
%
%
		\medskip
		
		{\it Step 5 ($C^{1,1}$ regularity):} Thanks to Lemma \ref{regmu}, we get that $f\in L^{\infty}$ with $\|f\|_{L^\infty}$ depending only on the Lipschitz character of $\partial E$ and on $\|f\|_{L^p}$. Using this new information,  we can improve \eqref{estimI1}, \eqref{estimIF} and \eqref{estimIN} to
		\begin{equation}\label{I3} I_1\les \eps^{2}, \quad I_2^F \les \gamma_\eps \eps^2, \quad \textrm{and} \quad I_2^N\les \gamma_\eps \eps^2 |\log \eps|.
		\end{equation} 
		Arguing as in {\it Step 4}, we find $\gamma_\eps \les Q \eps^{1/2}$
%
and thus $\partial E$ is of class $C^{1,1/2}$. In order to get higher regularity, we need to get a better estimate on $\gamma_\eps(x,y)$.

		Going back to \eqref{I1prep} and using  \eqref{estimgammaxy}  with $\beta=1/2$, we find the improved estimate
		\begin{equation}\label{I1best}
		I_1\les \eps^3.
		\end{equation}
		If we also use \eqref{estimgammaxy} in \eqref{definIN}, we obtain
		\begin{align*}
		I_2^N&\les \eps \gamma_\eps \int_{\Sigma_\eps\times \Gamma^N_\eps} \frac{\eps^{1/2}+ |x-y|^{1/2}}{\min_i |y-\tilde{x}_i^\eps|}
		\\
		&\les \eps \gamma_\eps \int_{\Sigma_\eps\times \Gamma^N_\eps} \frac{\eps^{1/2}+\min_i\{ |x-\tilde{x}_i^\eps|^{1/2}+ |y-\tilde{x}_i^\eps|^{1/2}\}}{\min_i |y-\tilde{x}_i^\eps|}\\
		&\les \eps \gamma_\eps \int_{\Sigma_\eps\times \Gamma^N_\eps} \frac{\eps^{1/2}+\min_i |y-\tilde{x}_i^\eps|^{1/2}}{\min_i |y-\tilde{x}_i^\eps|}\\
		&\les \eps^{2}\gamma_\eps\int_{ \Gamma^N_\eps} \frac{\eps^{1/2}}{\min_i |y-\tilde{x}_i^\eps|}+\frac{1}{\min_i |y-\tilde{x}_i^\eps|^{1/2}}\\
		&\les \eps^{2}\gamma_\eps (\eps^{1/2} |\log \eps| +1) \les \eps^2\gamma_\eps.
		\end{align*}
As in the beginning of Step 4, we can assume that  $\Delta V\ll\Delta P$, so that 
by \eqref{mainestimreg} and \eqref{DeltaP} we have  
$Q^2 \Delta \mathcal I_0\ges\Delta P\ges \varepsilon\gamma_\varepsilon^2$. 
By the previous estimate for $I_2^N$, \eqref{I1best} and the second inequality in \eqref{I3} we eventually get
\[
Q^2 \varepsilon^2\gamma_\varepsilon\sim Q^2(\varepsilon^3+\varepsilon^2\gamma_\varepsilon)\ges \varepsilon\gamma_\varepsilon^2,
\]
which leads to $\gamma_\eps\les Q^2 \eps$. By using again Lemma \ref{geomreglemma}, the proof is concluded.		
		\end{proof}
		
\section{Minimality of the ball for $N=2$ and $Q$ small}\label{secball}

We now use the regularity result obtained in Section \ref{secreg} to prove that for small charges, the only minimizers of $\mathcal F_{Q,0}$ in dimension two are balls.
\begin{theorem}\label{thm:minimalitapalla}
Let $N=2$ and $\alpha=0$.
There exists $Q_0>0$ such that for $Q<Q_0$, up to translations, the only   minimizer of \eqref{problema}  is the ball.	
\end{theorem}

\begin{proof}
Let $E_Q$ be a minimizer of $\mathcal F_{Q,0}$ and let $B$ be a ball of measure one. By minimality  of $E_Q$, we have
\begin{equation}\label{stimadeltaP}
P(E_Q)-P(B)\le Q^2\left( \Iz(B)-\Iz(E_Q)\right)\le Q^2\left( \Iz(B)+|\Iz(E_Q)|\right).
\end{equation}
By Lemma \ref{lemrego2} the diameter of $E_Q$ is uniformly bounded and so is $|\Iz(E_Q)|$.
Using  the quantitative isoperimetric inequality (see \cite{fumapr}), we infer
\[
 |E_Q\Delta B|^2\les P(E_Q)-P(B) \le Q^2\left( \Iz(B)+|\Iz(E_Q)|\right).
\]
This implies that $E_Q$ converges to $B$ in $L^1$  as $Q\to 0$. From the convexity of $E_Q$, this implies the convergence also in the Hausdorff metric. Since the sets $E_Q$ are all uniformly bounded and of fixed volume, they are uniformly  Lipschitz. 
Theorem \ref{mainregTh} then implies that  $\partial E_Q$ are  $C^{1,1}-$regular sets with $C^{1,1}$ norm  uniformly bounded. 
Therefore, thanks to the Arzel\`a-Ascoli's Theorem, we can write
\[
\partial E_Q=\left\{(1+\varphi_Q(x))x:\,x\in\partial B \right\},
\]
with $\|\varphi_Q\|_{C^{1,\beta}}$ converging to $0$ as $Q\to0$ for every $\beta<1$. From Lemma \ref{regmu}
  we infer that the optimal measures  $\mu_Q$ for $E_Q$  are uniformly $C^{0,\beta}$ and in particular are uniformly bounded. 
Using now \cite[Proposition 6.3]{gnrI}, we get that for small enough $Q$,
\[
 \|\mu_Q\|_{L^{\infty}}^2 \left(P(E_Q)-P(B)\right)\ges   \Iz(B)-\Iz(E_Q)
\]
Putting this into \eqref{stimadeltaP}, we then obtain
\[P(E_Q)-P(B)\les Q^2(P(E_Q)-P(B) )\]
from which we deduce that for $Q$ small enough, $P(E_Q)=P(B)$. Since, up to translations, the ball is the unique solution of the isoperimetric problem, this implies $E_Q=B$.
\end{proof}

\section{Asymptotic behavior as $Q\to +\infty$}\label{seclim}

In this section we characterize the limit shape of (suitably rescaled)
minimizers of $\FaQ$, with $\alpha\in [0,1]$, as the charge $Q$ 
tends to $+\infty$. For this, we fix a sequence $Q_n\to +\infty$.

\subsection{The case $\alpha\in [0,1)$}

For $n\in \N$, we let $V_n:=Q_n^{-\frac{2N(N-1)}{1+(N-1)\alpha}}$ (so that $V_n\to0$ as $n\to +\infty$) and 
\begin{eqnarray*}
\A_{n,\alpha} &:=& \left\{ E\subset\R^N\ \text{convex body,}
\ |E|=V_n\right\},
\\
\wFaQ(E) &:=& V_n^{-\frac{N-2}{N-1}}P(E)+\I_\alpha(E)
\qquad \textrm{ for } E\in\A_{n,\alpha}.
\end{eqnarray*}
\noindent
It is straightforward to check that if $E$ is a minimizer of  \eqref{problema}, 
then the rescaled set $$\widehat E:= Q_n^{-\frac{2(N-1)}{1+(N-1)\alpha}}\,E$$ 
is a minimizer of $\wFaQ$ in the class $\A_{n,\alpha}$. 

We begin with a compactness result for a sequence of sets
of equibounded energy.

\begin{proposition}\label{procomp}
Let $\alpha\in [0,1)$ and let $E_n\in \A_{n,\alpha}$ be
such that 
$$\sup_n \wFaQn(E_n)<+\infty.$$
Then, up to extracting a subsequence and up to rigid motions, 
the sets $E_n$ converge in the Hausdorff topology to the 
segment $[0,L]\times\{0\}^{N-1}$, for some $L\in (0,+\infty)$.
\end{proposition}

\begin{proof}
The bound on $\I_\alpha(E_n)$ directly implies with \eqref{estimlambdaalpha} (or \eqref{estimlambdazero} in the case $\alpha=0$) that the diameter of $E_n$
is uniformly bounded from below. 

Let us show that the diameter
of $E_n$ is also uniformly bounded from above. Arguing as in Theorem \ref{esistenza}, let $\mathcal{R}_n=\prod_{i=1}^N[0,\lambda_{i}^n]$ be the parallelepipeds given by Lemma \ref{john}, and assume without loss of generality that $\lambda_1^n\ge\lambda_2^n\ge\dots\ge\lambda_N^n$. In the case $\alpha>0$, \eqref{estimlambdageom} directly gives the bound 
while for $\alpha=0$, we get using \eqref{estimlambdageom} and \eqref{estimlambdazero}, that $|\Iz(\mathcal{R}_n)|$ is uniformly bounded, from which the bound on the diameter follows, using once again \eqref{estimlambdageom}. Moreover, from \eqref{estimlambdaalpha} and \eqref{estimlambdazero},
we obtain that $\lambda_i^n\sim  V_n^{\frac{1}{N-1}}$ (where the constants depend  on $\wFaQn(E_n)$), for $i=2,\dots,N$.
The convex bodies $E_n$ are therefore compact in 
the Hausdorff topology and 
any limit set is a non-trivial segment of length $L\in (0,+\infty)$.
\end{proof}
In the proof of the $\Gamma-$convergence result we will use the following result.
\begin{lemma}\label{concavprob}
 Let $0<\gamma<\beta$ with $\beta\ge 1$, $V>0$ and $L>0$, then 
 \begin{equation}\label{minconv}
  \min\lt\{ \int_{0}^L f^{\gamma} \ : \ \int_{0}^L f^\beta =V, \ f \textrm{ concave and } f\ge 0 \rt\}=\frac{(\beta+1)^{\gamma/\beta}}{\gamma+1} L^{1-\frac{\gamma}{\beta}} V^{\gamma/\beta}.
 \end{equation}
\end{lemma}

\begin{proof}
 For $L, V>0$, let 
 \[
 M(L,V):= \min\lt\{ \int_{0}^L f^{\gamma} \ : \ \int_{0}^L f^\beta =V, \ f \textrm{ concave and } f\ge 0 \rt\}.
 \]
Let us now prove \eqref{minconv}. By scaling, we can assume that $L=V=1$. Thanks to the concavity and positivity constraints, existence of a minimizer  for \eqref{minconv}  follows.  Let $f$ be such a minimizer. Let us prove that
we can assume that $f$ is non-increasing. Notice first that by definition, it holds
\[M(1,1)=\int_0^1 f^\gamma.\]
Up to a  rearrangement, we can assume that $f$ is symmetric around the point $1/2$, so that $f$ is non-increasing in $[1/2,1]$ and 
\[\int_{1/2}^1 f^\gamma =\frac{1}{2} M(1,1)=M(1/2,1/2).\]
Letting finally for $x\in[0,1]$, $\hat{f}(x):=f(\frac{1}{2}(x+\frac{1}{2}))$, we have that $\hat{f}$ is non-increasing, admissible for \eqref{minconv} and 
\[\int_0^1 \hat{f}^\gamma=2\int_{1/2}^1f^\gamma=M(1,1),\]
so that $\hat{f}$ is also a minimizer for \eqref{minconv}. \\
 Assume now that $f$ is not affine in $(0,1)$. Then there is $\overline x>0$ such that  for all $0<x\le\overline x$  
 \[f(x)> f(0)-(f(0)-f(1))x.\]
 Let $\tilde{f}:=\lambda -(\lambda-f(1))x$ with $\lambda>f(0)$ chosen so that 
 \begin{equation}\label{condvol}\int_0^1 f^{\beta-1} \tilde{f}=\int_0^1 f^{\beta}.\end{equation}
 Now, let $g:= \tilde{f}-f$. Since $f+g=\tilde{f}$ is concave, for every $0\le t \le 1$, $f+tg$ is a concave function. For  $\delta\in \R$, let $f_{t,\delta}:= f+t(g+\delta(1-x))$. Let finally $\delta_t$  be such that 
 \[\int_0^1 f_{t,\delta_t}^\beta=\int_0^1 f^\beta.\]
Thanks to \eqref{condvol} and since $\beta\ge 1$, $|\delta_t|=O(t)$.
%
Since $f_{t,\delta_t}$ is concave, by the minimality of $f$ we get 
 \[\int_0^1 f_{t,\delta_t}^\gamma- \int_0^1 f^\gamma\ge 0.\]
 Dividing by $t$ and taking the limit as $t$ goes to zero, we obtain
 \[
\int_0^1 f^{\gamma-1} g\ge 0.
 \]
Let  $z\in (0,1)$ be the unique point such that $\tilde f(z)=f(z)$ (so that $\tilde f(x) > f(x)$ for $x< z$ and $\tilde f(x) < f(x)$ for $x> z$). We then have,
\begin{align*}
0&\le  \int_0^1 f^{\beta-1} \frac{\tilde f-f}{f^{\beta-\gamma}}\\
&=\int_0^z f^{\beta-1} \frac{\tilde f-f}{f^{\beta-\gamma}}+\int_z^1 f^{\beta-1} \frac{\tilde f-f}{f^{\beta-\gamma}}\\
&< \frac{1}{f^{\beta-\gamma}(z)} \lt(\int_0^zf^{\beta-1} (\tilde f-f) +\int_z^1 f^{\beta-1} (\tilde f-f) \rt)\\
&= \frac{1}{f^{\beta-\gamma}(z)} \int_0^1 f^{\beta-1} (\tilde f-f),
\end{align*}
which contradicts \eqref{condvol}.\\

We are left to study the case when  $f$ is linear. Assume that $f(1)>0$ and let 
\[\delta:= \frac{\int_0^1 f^{\beta-1}}{\int_0^1 xf^{\beta-1}}>1,\]
so that in particular, $\int_0^1 f^{\beta-1} (1-\delta x)=0$.
Up to adjusting the volume as in the previous case, for $t>0$ small enough, $f+t(1-\delta x)$ is admissible. From this, arguing as above, we find that 
\[\int_0^1 f^{\gamma-1}(1-\delta x)\ge 0.\]
By splitting the integral around the point $\bar z=\delta^{-1}\in(0,1)$ and  proceeding as above,
we get again a contradiction. As a consequence, we obtain that $f(x)=\lambda (1-x)$, 
with $\lambda=(\beta+1)^{1/\beta}$ so that the volume constraint is satisfied. This concludes the proof of \eqref{minconv}.
\end{proof}

We now prove the following $\Gamma-$convergence result.
\begin{theorem}\label{gammaalpha}
For $\alpha\in [0,1)$,
the functionals $\wFaQ$ $\Gamma$--converge in the Hausdorff topology,
as $n\to +\infty$, to the functional
\[ 
\widehat{\mathcal{F}}_\alpha(E):=\begin{cases}
                            C_N \, L^{\frac{1}{N-1}}
+ \dfrac{\Ia([0,1])}{L^\alpha} 
& \textrm{if } E\simeq[0,L]\times\{0\}^{N-1} \textrm{ and }
\alpha>0
\\
\\
C_N \, L^{\frac{1}{N-1}}
+ \I_0([0,1]) - \log L
& \textrm{if } E\simeq[0,L]\times\{0\}^{N-1} \textrm{ and }
\alpha=0
\\
\\
                            +\infty & \textrm{otherwise,}
                           \end{cases}
\]
where $E\simeq F$ means that $E=F$ up to a rigid motion,
and $C_{N}:=\omega_{N-1}^{1/(N-1)}N^{(N-2)/(N-1)}$ with $\omega_N$ the volume of the ball of radius one in $\R^N$ (so that for $N=2$ we have $C_2=2$).
\end{theorem}

\begin{proof}
By Proposition \ref{procomp} we know that the $\Gamma$-limit
is $+\infty$ on the sets which are not segments.

Let us first prove the $\Gamma$-limsup inequality. 
Given $L\in (0,+\infty)$, we are going to construct $E_n$ symmetric with respect to the hyperplane $\{0\}\times\R^{N-1}$. For $t\in[0,L/2]$, we let $r(t):=\lt(\frac{NV_n}{\omega_{N-1} L}\rt)^{1/(N-1)}\lt(1-\frac{2t}{L}\rt)$ and then 
\[E_n\cap \lt(\R^+\times \R^{N-1}\rt):=\left\{ \left(t,B_{r(t)}^{N-1}\right) \ : t\in [0,L/2] \right\},\]
where $B_{r(t)}^{N-1}$ is the ball of radius $r(t)$ in $\R^{N-1}$. With this definition, $|E_n|=V_n$, so that $E_n\in \A_{n,\alpha}$. We then compute
\begin{align*}
P(E_n)&=2\int_0^{L/2} \mathcal H^{N-2}(\mathbb{S}^{N-2}) r(t)^{N-2}\sqrt{1+|r'|^2}
\\
&= 2(N-1)\,\omega_{N-1}\lt(\frac{NV_n}{\omega_{N-1} L}\rt)^\frac{N-2}{N-1}
\int_0^{L/2}\lt(1-\frac{2t}{L}\rt)^{N-2}\lt(1+ \frac{c_N}{L^2}\lt(\frac{V_n}{ L}\rt)^\frac{2}{N-1}\rt)^{1/2}
\\
&= C_N V_n^\frac{N-2}{N-1} L^\frac{1}{N-1} +o\left(V_n^\frac{N-2}{N-1}\right).
\end{align*}

Letting $\mu_\alpha$ be the optimal measure for 
$\I_\alpha([-L/2,L/2])$, we then have
\[
\wFaQ(E_{n}) \le C_N L^\frac{1}{N-1} + 
\I_\alpha([0,L]) +o(1),
\]
which gives the $\Gamma$-limsup inequality.

We  now turn to the  the $\Gamma$-liminf inequality. Let $E_n\in \A_{n,\alpha}$ be 
such that $E_n\to [0,L]\times\{0\}^{N-1}$ in the 
Hausdorff topology.
Since $\Ia$ is continuous under Hausdorff convergence, 
it is enough to prove that
 \begin{equation}\label{lowerbound}
 \liminf_{n\to+\infty} 
V_n^{-\frac{N-2}{N-1}}\,P(E_n)\ge C_{N}\, L^{\frac{1}{N-1}}.  
 \end{equation}
 Let $L_n:=\diam(E_n)$. By Hausdorff convergence, we have that $L_n\to L$. Moreover, up to a rotation and a translation, we can assume that $[0,L_n]\times\{0\}^{N-1}\subset E_n$. For $N=2$, we directly obtain $P(E_n)\ge 2L_n$ which gives \eqref{lowerbound}. We thus assume from now on that $N\ge 3$.
 Let $\widetilde{E}_n$ be the set
 obtained from $E_n$ after a Schwarz symmetrization around the axis $\R\times\{0\}^{N-1}$. By Brunn's principle \cite{Brunn}, $\widetilde{E}_n$ is still a convex set with $P(E_n)\ge P(\widetilde{E}_n)$ and $|E_n|= |\widetilde{E}_n|$. 
We thus have that 
\[
\widetilde{E}_n=\bigcup_{t\in[0,L_n]}\{t\}\times B_{r(t)}^{N-1}
\]
for an appropriate function $r(t)$, and, by Fubini's Theorem,
\[\int_0^{L_N} r(t)^{N-1}=\frac{V_n}{\omega_{N-1}}.\]
By the Coarea Formula \cite[Theorem 2.93]{AFP}, we then get 
\[
P(\widetilde{E}_n)\ge \mathcal H^{N-2}(\mathbb{S}^{N-2})\int_0^{L_n} r(t)^{N-2} \sqrt{1+|r'(t)|^2}\ge  \mathcal H^{N-2}(\mathbb{S}^{N-2})\int_0^{L_n} r(t)^{N-2}.
\]
Applying then Lemma \ref{concavprob} with $\gamma=N-2$ and $\beta=N-1$, we obtain \eqref{lowerbound}.
 \end{proof}

\begin{remark}\label{remass}\rm
For $\alpha\in[0,1)$ and $N\ge 2$, it is easy to optimize $\widehat{F}_{\alpha}$ in $L$ and obtain the values $L_{N,\alpha}$ given in Theorem \ref{teoconv}.
\end{remark}

From Proposition \ref{procomp},
Theorem \ref{gammaalpha} and the uniqueness of the minimizers for $\widehat{F}_{\alpha}$, we directly obtain the following asymptotic 
result for minimizers of \eqref{problema}.

\begin{corollary}\label{corasymptotic}
Let $\alpha\in [0,1)$ and $N\ge 2$. Then, up to rescalings and  rigid motions, every sequence $E_n$ of minimizers of \eqref{problema} converges in the
Hausdorff topology
to $[0,L_{N,\alpha}]\times\{0\}^{N-1}$. 
\end{corollary}

\subsection{The case $N=2,\,3$ and $\alpha=1$}
In the case $\alpha\ge 1$, the energy $\I_\alpha$ is infinite on segments and thus a $\Gamma-$limit of the same type as the one obtained in Theorem \ref{gammaalpha} cannot be expected. Nevertheless in the Coulombic case $N=3$, $\alpha=1$
we  can use a dual formulation of the non-local part of the energy to obtain the $\Gamma-$limit. As a by-product, we  can also treat the case  $N=2$, $\alpha=1$.

For $N=2,3$ and $n\in \N$, we let 
\begin{eqnarray*}
\A_{n,1} &:=& \left\{ E\subset\R^3\ \text{convex body,}
\ |E|=Q_n^{-2(N-1)}(\log Q_n)^{-(N-1)}\right\}\,,
\\
\wFQ(E) &:=& Q_n^{2(N-2)}(\log Q_n)^{N-2}\,P(E)+\frac{\I_1(E)}{\log Q_n},
\qquad  \textrm{for } E\in\A_{n,1}\,.
\end{eqnarray*}
\noindent
As before, if $E$ is a minimizer of \eqref{problema}, 
then the rescaled set 
$$
\widehat E:= Q_n^{-\frac{2(N-1)}{N}}(\log Q_n)^{-\frac{(N-1)}{N}}\,E
$$ 
is a minimizer of $\wFQ$ in $\A_{n,1}$. 

Let $C_\eps:=[0,1]\times B_\eps\subset \R^3$ be a narrow cylinder of radius $\eps>0$
(where $B_\eps$ denotes a two-dimensional ball of radius $\eps$).
We begin by proving the following estimate on $\I_1(C_\eps)$:

\begin{proposition}\label{proeps}
It holds that
\begin{equation}\label{estcap}
\lim_{\eps\to 0}  \frac{\I_1(C_\eps)}{|\log \eps|} =2\,.
\end{equation}
As a consequence, for every $L>0$,
\begin{equation}\label{estcapL}
\lim_{\eps\to 0}  \frac{\I_1([0,L]\times B_\eps)}{|\log \eps|} =\frac 2L\,.
\end{equation}
\end{proposition}

\begin{proof}
The equality in \eqref{estcap} is well-known (see for instance \cite{maxwell}).
We include here a proof for the reader's convenience.

To show that
\[\lim_{\eps\to 0}  |\log \eps|^{-1} \I_1(C_\eps) \le 2,\] 
we use  $\mu_\eps:=\frac{1}{\pi \eps^2} \chi_{C_\eps}$ as  a test measure in the definition of $\I_1(C_\eps)$. Then, noting that for every $y\in C_\eps$,
\[\int_{C_\eps+y} \frac{dz}{|z|}\le \int_{[-1/2,1/2]\times B_\eps} \frac{dz}{|z|},\]
we obtain
\begin{align*}
\I_1(C_\eps)&\le \frac{1}{\pi^2\eps^4} \int_{C_\eps\times C_\eps}\frac{dxdy}{|x-y|}=\frac{1}{\pi^2\eps^4} \int_{C_\eps}\lt(\int_{C_\eps+y} \frac{dz}{|z|}\rt) dy\\
&\le \frac{1}{\pi\eps^2}\int_{-1/2}^{1/2} \int_{B_\eps} \frac{1}{(z_1^2+|(z_2,z_3)|^2)^{1/2}}
=\frac{4}{\eps^2}\int_{0}^{1/2} \int_{0}^{\eps} \frac{r}{(z_1^2+r^2)^{1/2}}\\
&=\frac{4}{\eps^2}\int_{0}^{1/2} \sqrt{z_1^2+\eps^2}-z_1 \\
&=\frac{4}{\eps^2}\lt( \frac{1}{8}\sqrt{1+4 \eps^2}-\frac{1}{8} +\frac{\eps^2}{2}\log\lt(\frac{1}{2\eps}+\sqrt{1+\frac{1}{4\eps^2}}\rt)\rt)\\ 
&= 2|\log \eps|+o(|\log \eps|).
\end{align*}

In order to show the opposite inequality, we recall the following definition of capacity of a set $E$:  
\[
\Cap(E):=\min \left\{\int_{\R^3} |\nabla \phi|^2 \ : \ \chi_E\le \phi, \phi\in H^1_0(\R^3)\right\}
\]
Then , if $E$ is compact, we have \cite{landkof,gnrI}
\[\I_1(E)=\frac{4\pi}{\Cap(E)}.\]
Thus \eqref{estcap} will be proved once we show that 
\begin{equation}\label{estimCap}
\Cap(C_\eps) |\log \eps|\le 2\pi +o(1).\end{equation}
For this, let $\lambda>0$ and $\mu>0$ to be fixed later and let 
\[f_\lambda(x'):=\begin{cases}
                  1 & \textrm{for } |x'|\le \eps\\
                  1-\dfrac{\log(|x'|/\eps)}{\log(\lambda/\eps)} & \textrm{for } \eps\le |x'|\le \lambda\\
                  0 & \textrm{for } |x'|\ge \lambda
                 \end{cases}\]
                 
and
\[\rho_\mu(z):=\begin{cases}
               0 & \textrm{for } z\le-\mu\\
               \dfrac{z+\mu}{\mu} & \textrm{for } -\mu\le z\le 0\\
               1 &\textrm{for } 0\le z\le 1\\
               1-\dfrac{z-1}{\mu} &\textrm{for } 1\le z\le 1+\mu\\
               0& \textrm{for } z\ge 1+\mu.\\
               \end{cases}\]
 We finally let $\phi(x',z):= f_\lambda(x')\rho_\mu(z)$.
 Since $\rho_\mu, f_\lambda\le 1$ and $|\rho_\mu'|\le \mu^{-1}$, by definition of $\Cap(C_\eps)$, we have 
 \begin{align*}
  \Cap(C_\eps)&\le \int_0^1 \frac{2\pi}{\log(\lambda/\eps)^2} \int_\eps^\lambda \frac{1}{r} +C\lt(\frac{\mu}{\log(\lambda/\eps)}+\frac{\lambda^2}{\mu}\rt)\\
  &\le \frac{2\pi}{\log(\lambda/\eps)} +C\lt(\frac{\mu}{\log(\lambda/\eps)}+\frac{\lambda^2}{\mu}\rt). 
 \end{align*}
We now choose $\lambda:=|\log \eps|^{-1}\gg \eps$ and $\mu:=|\log \lambda|^{-1}=(\log |\log \eps|)^{-1}$ so that $\log(\lambda/\eps)= |\log \eps|+\log |\log \eps|$, $\mu\to 0$ and $\mu\gg \lambda$ so that 
\[\frac{\mu}{\log(\lambda/\eps)}+\frac{\lambda^2}{\mu}=o(|\log \eps|^{-1})\]
and we find \eqref{estimCap}.\\
The equality in \eqref{estcapL} then follows by scaling.
\end{proof}
As a simple corollary we get the two dimensional result 
\begin{corollary}

 \begin{equation}\label{estcap2d}
  \lim_{\eps\to 0}  \frac{\I_1([0,1]\times[0,\eps])}{|\log \eps|} =2\,.
 \end{equation}

\end{corollary}
\begin{proof}
 The upper bound is obtained as above by testing with $\mu_\eps:= \eps^{-1} \chi_{[0,1]\times[0,\eps]}$. By identifying  
 $[0,1]\times[0,\eps]$ with $[0,1]\times[0,\eps]\times\{0\}\subset C_\eps$ we get that $\I_1([0,1]\times[0,\eps])\ge \I_1(C_\eps)$. This 
 gives together with \eqref{estcap} the corresponding lower bound.
\end{proof}
We can now prove a compactness result analogous to Proposition \ref{procomp}.

\begin{proposition}\label{procompuno}
Let $E_n\in \A_{n,1}$ be 
such that $\sup_n \wFQ(E_n)<+\infty$.
Then, up to extracting a subsequence and up to rigid motions, 
the sets $E_n$ converge in the Hausdorff topology to a 
segment $[0,L]\times\{0\}^{N}$, for some $L\in (0,+\infty)$.
\end{proposition}

\begin{proof}
We argue as in the proof of Proposition \ref{procomp}. Since the case $N=2$ is easier, we focus on $N=3$.
Let $\mathcal{R}_n=\prod_{i=1}^3[0,\lambda_{i,n}]$ be given by Lemma \ref{john} and let us assume without loss of generality that  $i\mapsto\lambda_{i,n}$ is decreasing. 
Then \eqref{estimlambdageom} applied with $V=Q_n^{-4} (\log Q_n)^{-2}$, 
directly yields an upper bound on $\lambda_{1,n}$ (and thus on $\diam (E_n)$). 

We now show that the diameter of $E_n$ is also uniformly bounded from 
below. Unfortunately,   \eqref{estimlambdaalpha} does not give the right bound and we  need to refine it using \eqref{estcap}.  As in Proposition \ref{procomp}, the energy bound 
$\I_1(E_n)\les  \log Q_n$, directly implies that 
\[
\lambda_{1,n}\ges\frac{1}{\log Q_n}\,,
\]
from which, using \eqref{estimlambdageom} and $\prod_{i=1}^3 \lambda_{i,n}\sim Q_n^{-4} (\log Q_n)^{-2}$, we get
\[
\lambda_{2,n}\les Q_n^{-2}\,.
\]
In particular, it follows that 
$$
\frac{\lambda_{2,n}}{\lambda_{1,n}}\les \frac{\log Q_n}{Q_n^2}\,.
$$
By Proposition \ref{proeps}, letting $\eps_n := Q_n^{-2} \log Q_n$ we get
\begin{eqnarray*}
\lambda_{1,n} \log Q_n &\ges& \lambda_{1,n}\I_1(E_n)
\sim \lambda_{1,n} \I_1\left( \mathcal{R}_n\right) 
\\
&=& \I_1\left( \prod_{i=1}^3 \left[0,
\frac{\lambda_{i,n}}{\lambda_{1,n}}\right]\right) \ges 
\I_1\left( C_{\eps_n}\right)
\\
&\sim& |\log \eps_n| \sim \log Q_n\,,
\end{eqnarray*}
which implies
$$
\lambda_{1,n}\ges 1\,,
$$
and gives a lower bound on the diameter of $E_n$.

Arguing as in the proof of \eqref{estimlambdaalpha}, we  then get
\begin{equation}\label{stimalambdabis}
\lambda_{3,n}\leq \lambda_{2,n} \les
Q_n^{-2}(\log Q_n)^{-1}\,.
\end{equation}

It follows that the sets $E_n$ are compact in 
the Hausdorff topology, and any limit set is a segment 
of length $L\in (0,+\infty)$.
\end{proof}

Arguing as in Theorem \ref{gammaalpha}, we obtain the following result.

\begin{theorem}\label{gamma1}
The functionals $\wFQ$, $\Gamma$--converge 
in the Hausdorff topology, to the functional
 \[ \widehat{\mathcal{F}}_1(E):=\begin{cases}
                            C_N \, L^{\frac{1}{N-1}}
+ \dfrac{4}{L} 
& \textrm{if } E\simeq[0,L]\times\{0\}^{N-1} 
\\
                            +\infty & \textrm{otherwise,}
                           \end{cases}
\]
where $C_N$ is defined as in Theorem \ref{gammaalpha}.
\end{theorem}

\begin{proof}
 Since the case $N=2$ is easier, we focus on $N=3$. The compactness and lower bound for the perimeter are obtained  exactly 
as in Theorem \ref{gammaalpha}. For the upper bound, 
for $L>0$ and $n\in \N$, we define $E_n$ as in the proof of Theorem \ref{gammaalpha}, by first letting $V_n:=Q_n^{-4}(\log Q_n)^{-2}$ (recall that $N=3$) and then  for $t\in[0,L/2]$, $r(t):=\lt(\frac{3 V_n}{\pi L}\rt)^{1/2}\lt(1-\frac{2t}{L}\rt)$ and 
\[
E_n\cap \lt(\R^+\times \R^{2}\rt):=\bigcup_{t\in[0,\frac L2]}\{t\}\times B_{r(t)}^{2} 
\]
where $B_{r(t)}^{2}$ is the ball of radius $r(t)$ in $\R^{2}$. 

As in the proof of Theorem \ref{gammaalpha}, we have
\[
\lim_{n\to +\infty} Q_n^{2} \log Q_n \,P(E_n)= C_3 \, L^{\frac{1}{2}}.
\]
 Let $\mu_n$ be the optimal measure for $\mathcal{I}_1(E_n)$, and let $\eps_n:=\lt(\frac{3 V_n}{\pi L}\rt)^{1/2}$. For  $L>\delta>0$, $[-\frac{L-\delta}{2},\frac{L-\delta}{2}]\times B_{\eps_n}^2\subset E_n$  so that by \eqref{estcapL},
 \[\mathcal{I}_1(E_n)\le \mathcal{I}_1\lt(\lt[-\frac{L-\delta}{2},\frac{L-\delta}{2}\rt]\times B_{\eps_n}^2\rt)=\frac{|\log V_n|}{(L-\delta)} +o(|\log V_n|). 
 \]
  Recalling that $|\log V_n|= 4 |\log Q_n| +o(|\log Q_n|)$, we then get
 \[\varlimsup_{n\to +\infty} \frac{\mathcal{I}_1(E_n)}{\log(Q_n)}\le \frac{4}{L-\delta}.\]
Letting $\delta\to 0^+$, we obtain the upper bound.\\
 
We are left to prove the lower bound for the non-local part of the energy. Let  $E_n$ be   be a sequence of convex sets such that $E_n\to [0,L]\times\{0\}^2$ and such that $|E_n|=Q_n^{-4} (\log Q_n)^{-2}$. We can assume that $\sup_n \wFQ(E_n)<+\infty$, since otherwise there is nothing to prove. 
Let $\delta>0$. Up to a rotation and a translation, we can assume that $[0,L-\delta]\times\{0\}^2\subset E_n \subset [0,L+\delta] \times\R^2$ for $n$ large enough. Let now $x^1=(x^1_1,x^1_2,x^1_3)$ be such that 
\[|(x^1_2,x^1_3)|=\max_{x\in E_n} |(x_2,x_3)|.\]
Up to a rotation of axis $\R \times \{0\}^2$, we can assume that $x^1=(a,\ell^n_1,0)$ for some $\ell^n_1\ge 0$. Let finally $x^2$ be such that  
\[|x^2\cdot e_3|=\max_{x\in E_n} |x\cdot e_3|\]
so that $x^2=(b,c,\ell^n_2)$ with $\ell^n_2\le \ell^n_1$. Since by definition $E_n\subset [0,L+\delta]\times [-\ell^n_1,\ell^n_1]\times[-\ell^n_2,\ell^n_2]$, we have $Q_n^{-4}(\log Q_n)^{-2}=|E_n|\les  \ell^n_1 \ell^n_2 (L+\delta)$. 
On the other hand, by convexity, the tetrahedron $T$ with vertices $0$, $x_1$, $x_2$ and $(L-\delta, 0, 0)$ is contained in $E_n$. We thus have $|E_n|\ge |T|$. Since
\[|T|=\frac{1}{8} |\det (x^1,x^2,(L-\delta, 0, 0))|=\frac{1}{8} (L-\delta)\ell^n_1 \ell^n_2,\]
we also have $Q_n^{-4}(\log Q_n)^{-2}\ges  \ell^n_1 \ell^n_2 (L-\delta)$. Arguing as in the proof of \eqref{estimlambdaalpha}, we get from the energy bound, $(L-\delta) \ell_1^n\les Q_n^{-2} (\log Q_n)^{-1}$, and thus 
\[  \ell^n_1 \ell^n_2\ges \frac{1}{ (L-\delta) Q_n^4 (\log Q_n)^2}.\]
From this we get $\ell^n_1\sim\ell^n_2\sim Q_n^{-2} (\log Q_n)^{-1}$, where the constants involved might depend on $L$. We therefore have $E_n\subset  [0,L+\delta]\times B_{C Q_n^{-2} (\log Q_n)^{-1}}$ for $C$ large enough. From this we infer that
\begin{align*}
 \liminf _{n\to +\infty}\frac{\I_1(E_n)}{\log Q_n}&\ge\liminf _{n\to +\infty} \frac{\I_1([0,L+\delta]\times B_{C Q_n^{-2} (\log Q_n)^{-1}})}{\log Q_n} \\
 &\ge 2  \liminf _{n\to +\infty} \frac{\I_1([0,L+\delta]\times B_{C Q_n^{-2} (\log Q_n)^{-1}})}{\log (C Q_n^{-2} (\log Q_n)^{-1})}\\
 &\ge 4(L+\delta)^{-1},
\end{align*}
where the last inequality follows from \eqref{estcapL}.
Letting $\delta\to 0$, we conclude the proof.
\end{proof}

\begin{remark}
 As before, optimizing $ \widehat{\mathcal{F}}_1$ with respect to $L$,
 one easily obtains the values of $L_{N,1}$ given in Theorem \ref{teoconv}.
\end{remark}

\begin{remark}
 By analogy with results obtained in  the setting of minimal Riesz energy point configurations \cite{hardinsaff,saffmartinez}, we believe that for every $N\ge 2$, $\alpha>1$ and $L>0$, 
 \eqref{estcapL} can be generalized to 
 \begin{equation}\label{estcapgene}\lim_{\eps \to 0} \frac{\I_\alpha([0,L]\times[0,\eps]^{N-1})}{\eps^{1-\alpha}} =\frac{C_\alpha}{L^{\alpha}},\end{equation}
 for some constant $C_\alpha$ depending only on $\alpha$. This result would  permit one to extend Theorem \ref{gamma1} beyond $\alpha=1$. 
 Let us point out that showing that the right-hand side of \eqref{estcapgene} is bigger than the left-hand side can be easily obtained by plugging in the uniform measure as  a test measure. 
 However, we are not able to prove the reverse inequality.  
\end{remark}


\bibliographystyle{plain}

\end{document}